\documentclass[10pt]{article}
\usepackage{amsmath,amssymb}
\usepackage{amsthm}
\usepackage{cite}
\numberwithin{equation}{section}
\usepackage{graphicx}
\usepackage{float}
\usepackage[colorlinks=true,linkcolor=black]{hyperref}
\hypersetup{citecolor=black}
\usepackage[left=1in,right=0.5in,top=0.5in,bottom=0.5in,footskip=0.1in]{geometry}

\newtheorem{theorem}{Theorem}[section]

\newtheorem{lemma}[theorem]{Lemma}

\title{On a seventh order convergent weakly $L$-stable Newton Cotes formula with application on Burger's equation}
\author{Amit Kumar Verma$^a$\thanks{$^a$akverma@iitp.ac.in, $^c$ravi.agarwal@tamuk.edu }, Mukesh Kumar Rawani$^b$, Ravi P. Agarwal$^c$\\{\small{\it{$^{a,b}$Department of Mathematics, Indian Institute of Technology Patna,}}}\\{\small{\it{Patna, Bihar 801106, India.}}}\\\small{\it{$^c$Department of Mathematics,Texas A\&M, University-Kingsville,}}\\\small{\it{ 700 University Blvd., MSC 172, Kingsville, Texas  78363-8202.}}}

\begin{document}
	
\maketitle
	
\begin{abstract}
In this paper we derive $7^{th}$ order convergent integration formula in time which is weakly $L$-stable. To derive the method we use, Newton Cotes formula, fifth-order Hermite interpolation polynomial approximation (osculatory interpolation) and sixth-order explicit backward Taylor's polynomial approximation. The vector form of this formula is used to solve Burger's equation which is one dimensional form of Navier-Stokes equation. We observe that the method gives high accuracy results in the case of inconsistencies as well as for small values of viscosity, e.g., $10^{-3}$. Computations are performed by using Mathematica 11.3. Stability and convergence of the schemes are also proved. To check the efficiency of the method we considered 6 test examples and several tables and figures are generated which verify all results of the paper.
\end{abstract}
\section {Introduction}
Burgers' equation with $\nu_d$ as coefficient of viscosity can be defined as
\begin{equation}
\label{lb1}
\frac{\partial w}{\partial t}+w\frac{\partial w}{\partial x}-\frac{\nu_{d}}{2}\frac{\partial^{2}w}{\partial x^{2}}=0,\hspace{.5cm}(x,t)\in\Sigma_{T},
\end{equation}
where
\begin{equation*}
\Sigma_{T}=\left(\alpha_{0},\alpha_{1}\right) \times\left(0,T\right],~ T> 0,
\end{equation*}
with Dirichlet boundary conditions (BCs),
\begin{equation}
\label{bc}
w(\alpha_{i},t)=0,~~i=0,1 ~~\mbox{and} ~~t\in(0,T],
\end{equation}
and  initial conditions (ICs),
\begin{equation}
w(x,0)=f(x), \hspace{.5cm} x\in\left(\alpha_{0},\alpha_{1}\right).
\end{equation}

Linearized form of Burgers' equation (by using Hopf-Cole transformation) is given as
 \begin{equation}  \frac{\partial\psi}{\partial{t}}=\frac{\nu_{d}}{2}\frac{\partial^{2}\psi}{\partial{x}^{2}}
 \end{equation}
 with the Neumann boundary conditions $(BCs),$ 
 \begin{equation*}
 	\psi_{x}(\alpha_{i},t)=0,~~ i=0,1,
 \end{equation*}
 and the initial conditions $(ICs),$
 \begin{equation*}
 	\psi(x,0)=g(x).
 \end{equation*}

The study of Burger's equation is popular among the scientific community as it is very simple form of Naviers' Stokes equation and due to it's appearence in various field  of applied mathematics and physics such as in the context of gas dynamics, in the theory of shock waves, traffic flow, mathematical modeling of turbulent fluid and in continuous stochastic processes.
In $1915$, it  was first introduced by Bateman \cite{bateman1915some}. Later in $1948$, it was introduced by Burger \cite{burgers1939mathematical,burgers1948mathematical} as a class of equation which delineate the mathematical model of turbulence. Recently in $2019$, Ryu etc. \cite{ryu2019improved} propose some nowcasting rainfall models based on Burger's equation.  This equation has been solved analytically for some initial condition and solution is represented in the form of Fourier series expansion which converges slowly for small values of viscosity. Exact solution does not work very well for small values of viscosity and hence it always attracts researchers to test newly devloped numerical method on this nonlinear parabolic PDEs.
	\par Recently, with development in computer speed several numerical schemes based on finite difference method, finite element method, spectral method, differential quadrature method, decomposition method, moving least squares particle method, Haar wavelet quasilinearization approach etc., have been developed to solve the Burger's equation
\cite{caldwell1982solution,saka2007quartic,aksan2006quadratic,ozics2003finite,mittal1993numerical,hassanien2005fourth,xu2011novel,kutluay1999numerical,asaithambi2010numerical,dogan2004galerkin,ali1992collocation,khater2008chebyshev,korkmaz2011polynomial,korkmaz2011quartic,mittal2012differential,mittal2009differential,fu2019moving,jiwari2013numerical,jiwari2015hybrid,bakodah2017decomposition,gowrisankar2019efficient,seydaouglu2019meshless,shiralashetti2019numerical,seydaouglu2018accurate,elgindy2018high,lukyanenko2018solving, KPLVAKVAJOM, AMC2009}. 

Crank-Nicolson (CN) method \cite{crank1947practical,smith1978numerical,thomas2013numerical,lawson1978extrapolation} is a second order method which is based on Trapezoidal formula which is A-stable but not L-stable. In the presence of inconsistencies \cite{morton1967difference} CN produces unwanted oscillations. Chawla etc. \cite{chawla1999generalized} produces generalised Trapezoidal formula (GTF($\alpha$)), where $\alpha>0$ which is L-stable and gives a quite stable result. Chawla etc. \cite{chawla1994stabilized} proposed a modified Simpson's $1/3$ rule (ASIMP) which is A-stable and used to give fourth-order time integration formula but it produces unwanted oscillations like CN due to lack of L-stability. To remove this oscillation, Chawla etc. \cite{chawla2005new} produced  L-stable version of Simpson's $1/3$ rule and implemented it to derive a third-order time integration formula for  the diffusion equation which gives stable and accurate result. Lajja \cite{verma2012stable} proposed L-stable derivative free error corrected Trapezoidal rule (LSDFECT). Verma etc. \cite{verma2015higher} developed a fifth order time integration formula for the diffusion equation which is weakly L-stable. 

Here we derive $7^{th}$ order time integration formula which is weakly $L$-stable and generalize above mentioned existing results.  The issue of slow convergence of series solution for small $\nu_d$ forces analytical solution of Burgers' equation to deviate from the exact solution. So, it is not easy to compute the solution for small values of $\nu_d$. The newly developed method computes the solution even for small values of $\nu_d$. To compute the numerical solution we use Mathematica 11.3 and find out that numerical solutions are in good agreement for small values of $\nu_{d}$. The result are in good agreement with exact solution when inconsistencies are present in the initial and boundary condition. 

The paper is organized as follows. In section 2, we give close form solution which we use to compute exact solution. In section 3, we derive higher order integration method in time for $u'(t)=f(t,u)$. In section 4, we use this technique  combined with finite difference to solve Burgers' equation and demonstrate the stability. In section 5, we illustrate the numerical results with tables and 2D-3D graphs.  

\section{Close Form Solution}
	Hopf \cite{hopf1950partial} and Cole \cite{cole1951quasi}  gave idea that the equation \eqref{lb1} can be reduced to the following linear heat equation	
	\begin{equation} 
	\label{lb2}
	 \frac{\partial\psi}{\partial{t}}=\frac{\nu_{d}}{2}\frac{\partial^{2}\psi}{\partial{x}^{2}},
	\end{equation}
	with the Neumann boundary condition (BC)
	\begin{equation}
	\psi_{x}(\alpha_{i},t)=0,~~ \alpha_{i}=i,~~ i=0,1,
	\end{equation}
	and the initial condition (IC)
	\begin{equation}
	\psi(x,0)=g(x),
\end{equation}
by non-linear transformation 
	\begin{equation}
	\label{r1}
\phi=-\nu_{d}(log\psi), \hspace{.5cm} \phi=\phi(x,t),
\end{equation}
	and
\begin{equation}
\label{r2}
w=\phi_{x}.
\end{equation}
The analytical solution of the linearized heat equation \eqref{lb2} is given by 
\begin{equation}
\label{lb8}
\psi(x,t)=\beta_{0}+\sum_{l=1}^{\infty}\beta_{l}\exp\left(-\frac{\nu_{d} l^2\pi^2 t}{2}\right)\cos (l\pi x),
\end{equation}
where $\beta_{0}$ and $\beta_{l}$ are Fourier coefficient and is given by

\begin{eqnarray*}
	\beta_{0}&=&\int_{0}^1 \exp\left(-\frac{1}{\nu_{d}}\int_{0}^xw_{0}(\xi)d\xi\right) dx,\\
	\beta_{l}&=&2\int_{0}^1 \exp\left(-\frac{1}{\nu_{d}}\int_{0}^xw_{0}(\xi)d\xi\right)\cos(l\pi x) dx,
\end{eqnarray*}
where $w_{0}(\xi)=w(\xi,0) $. 

The analytical solution by Hopf-Cole transformation is

\begin{equation}
\label{ANBE}w(x,t)=\pi \nu_{d}\frac{\sum_{l=1}^{\infty}\beta_{l}\exp(-\frac{\nu_{d} l^2\pi^2 t}{2})l\sin (l\pi x)}{\beta_{0}+\sum_{l=1}^{\infty}\beta_{l}\exp(-\frac{\nu_{d} l^2\pi^2 t}{2})\cos (l\pi x)}.
\end{equation}

\section{Illustration of the proposed method}
We consider the initial value problem
\begin{equation}
u'(t)=f(t,u),\hspace{.5cm} u(t_{0})=\eta_{0}.
\end{equation}
The Newton Cotes time integration formula is given by
\begin{equation}
\label{lb3}
u_{n+1}=u_{n}+\frac{h}{840}\Big( 41f_{n}+216f_{n+1/6}+27f_{n+2/6}+272f_{n+3/6}+27f_{n+4/6}+216f_{n+5/6}+41f_{n+1}\Big).
\end{equation}
Now, we use the fifth order Hermite approximation for $u_{n+1/6}$, $u_{n+2/6}$, $u_{n+3/6}$, $u_{n+4/6}$, $u_{n+5/6}$ which are  given by 
\begin{eqnarray}
&& u_{n+1/6}=\frac{1}{15552}\Big(1500y_{n}+552u_{n+1}+2250 hu'_{n}-210 hu'_{n+1}+125h^2u''_{n}+25h^2u''_{n+1}\Big),\\
&&u_{n+2/6}=\frac{1}{243}\Big(192u_{n}+51u_{n+1}+48 hu'_{n}-18 hu'_{n+1}+4h^2u''_{n}+2h^2u''_{n+1}\Big),\\
&&u_{n+3/6}=\frac{1}{64}\Big(32u_{n}+32u_{n+1}+10 hu'_{n}-10 hu'_{n+1}+h^2u''_{n}+h^2u''_{n+1}\Big),\\
&&u_{n+4/6}=\frac{1}{243}\Big(51u_{n}+192u_{n+1}+18 hu'_{n}-48 hu'_{n+1}+2h^2u''_{n}+4h^2u''_{n+1}\Big),\\
&& u_{n+5/6}=\frac{1}{15552}\Big(552u_{n}+15000u_{n+1}+210 hu'_{n}-2250 hu'_{n+1}+25h^2u''_{n}+125h^2u''_{n+1}\Big),
\end{eqnarray}
and  sixth order Taylor's approximation
\begin{eqnarray}
&&\label{l6}
u_{n}=u_{n+1}-hu'_{n+1}+\frac{h^2}{2}u''_{n+1}-\frac{h^3}{6}u'''_{n+1}+\frac{h^4}{24}u^{iv}_{n+1}-\frac{h^5}{120}u^v_{n+1},
\end{eqnarray}
to get
\begin{eqnarray}
\nonumber\overline {u_{n+1/6}}&=&\frac{1}{46656}\Big[  44875u_{n}+1781u_{n+1}+6750hu'_{n}+375h^2u''_{n}-755hu'_{n+1}+\frac{275}{2}h^2u''_{n+1}\\
&&+125(-\frac{h^3}{6}u'''_{n+1}-\frac{h^4}{6}u^{iv}_{n+1}-\frac{h^5}{6}u^{v}_{n+1})\Big],\\
\nonumber\overline {u_{n+2/6}}&=&\frac{1}{729}\Big[  568u_{n}+161u_{n+1}+144hu'_{n}+12h^2u''_{n}-62hu'_{n+1}+10h^2u''_{n+1}+\\
&&8(-\frac{h^3}{6}u'''_{n+1}-\frac{h^4}{6}u^{iv}_{n+1}-\frac{h^5}{6}u^{v}_{n+1})\Big],\\
\nonumber \overline {u_{n+3/6}}&=&\frac{1}{64}\Big[  31u_{n}+33u_{n+1}+10hu'_{n}+h^2u''_{n}-11hu'_{n+1}+\frac{3}{2}h^2u''_{n+1}+(-\frac{h^3}{6}u'''_{n+1}-\\
&&\frac{h^4}{6}u^{iv}_{n+1}-\frac{h^5}{6}u^{v}_{n+1})\Big],
\end{eqnarray}
\begin{eqnarray}
\nonumber\overline {u_{n+4/6}}&=&\frac{1}{729}\Big[  145u_{n}+584u_{n+1}+54hu'_{n}+6h^2u''_{n}-152hu'_{n+1}+20h^2u''_{n+1}+\\&&8(-\frac{h^3}{6}u'''_{n+1}-
\frac{h^4}{6}u^{iv}_{n+1}-\frac{h^5}{6}u^{v}_{n+1})\Big],\\
\nonumber\overline{u_{n+5/6}}&=&\frac{1}{46656}\Big[  1531u_{n}+45125u_{n+1}+630hu'_{n}+75h^2u''_{n}-6875hu'_{n+1}+ \frac{875}{2}h^2u''_{n+1}\\
&&+125(-\frac{h^3}{6}u'''_{n+1}-\frac{h^4}{6}u^{iv}_{n+1}-\frac{h^5}{6}u^{v}_{n+1})\Big].
\end{eqnarray}
Now, we define
\begin{eqnarray}
&&\overline{f_{n+1/6}}=f(x_{n+1/6},\overline{u_{n+1/6}}),\\
&&\overline{f_{n+2/6}}=f(x_{n+2/6},\overline{u_{n+2/6}}),\\
&&\overline{f_{n+3/6}}=f(x_{n+3/6},\overline{u_{n+3/6}}),\\
&&\overline{f_{n+4/6}}=f(x_{n+4/6},\overline{u_{n+4/6}}),\\
&&\overline{f_{n+5/6}}=f(x_{n+5/6},\overline{u_{n+5/6}}).
\end{eqnarray}
Now, the time integral formula \eqref{lb3} for the interval $[t_{n},t_{n+1}]$ becomes
\begin{eqnarray}
\label{lb4}
 u_{n+1}=u_{n}+\frac{h}{840}\Big( 41f_{n}+216\overline{f_{n+1/6}}+27\overline{f_{n+2/6}}+272\overline{f_{n+3/6}}+27\overline{f_{n+4/6}}
+216\overline{f_{n+5/6}}+41f_{n+1}\Big).
\end{eqnarray}
\subsection{Local trunction error}

Using Taylor's series expansion, we have
\begin{eqnarray}
\nonumber	u_{n+1/6}&=&\frac{1}{15552}\Big[1500u_{n}+552u_{n+1}+2250 hu'_{n}-210 hu'_{n+1}+125h^2u''_{n}\\
	&&+25h^2u''_{n+1}\Big]-\frac{25h^6}{6718464}u^{vi}_{n}-\frac{475h^7}{282175488}u^{vii}_{n}+\mathcal{O}(h^8),\\
\nonumber	u_{n+2/6}&=&\frac{1}{243}\Big[192u_{n}+51u_{n+1}+48 hu'_{n}-18 hu'_{n+1}+4h^2u''_{n}+2h^2u''_{n+1}\Big]\\
	&&-\frac{h^6}{65610}u^{vi}_{n}-\frac{h^7}{137781}u^{vii}_{n}+\mathcal{O}(h^8),\\
\nonumber	u_{n+3/6}&=&\frac{1}{64}\Big[32u_{n}+32u_{n+1}+10 hu'_{n}-10 hu'_{n+1}+h^2u''_{n}+h^2u''_{n+1}\Big]\\
	&&-\frac{h^6}{46080}u^{vi}_{n}-
	\frac{h^7}{92160}u^{vii}_{n}+\mathcal{O}(h^8),\\
\nonumber	u_{n+4/6}&=&\frac{1}{243}\Big[51u_{n}+192y_{n+1}+18 hu'_{n}-48 hu'_{n+1}+2h^2u''_{n}+4h^2u''_{n+1})\Big]\\
	&&-\frac{h^6}{65610}u^{vi}_{n}-\frac{11h^7}{1377810}u^{vii}_{n}+\mathcal{O}(h^8),\\
\nonumber	u_{n+5/6}&=&\frac{1}{15552}\Big[552u_{n}+15000u_{n+1}+210 hu'_{n}-2250 hu'_{n+1}+25h^2u''_{n}\\ &&+125h^2u''_{n+1}\Big]-\frac{25h^6}{6718464}u^{vi}_{n}-\frac{575h^7}{282175488}u^{vii}_{n}+\mathcal{O}(h^8),\\
\nonumber	u_{n}&=&u_{n+1}-hu'_{n+1}+\frac{h^2}{2}u''_{n+1}-\frac{h^3}{6}u'''_{n+1}+\frac{h^4}{24}u^{iv}_{n+1}-\frac{h^5}{120}u^v_{n+1}+\frac{h^6}{720}u^{vi}_{n}\\
	&&+\frac{h^7}{840}u^{vii}_{n}+\mathcal{O}(h^8),
\end{eqnarray} 
then it follows that
\begin{eqnarray}
&&u_{n+1/6}=\overline{u_{n+1/6}}+\frac{425h^7}{282175488}u^{vii}_{n}+\mathcal{O}(h^8),\\
&&u_{n+2/6}=\overline{u_{n+2/6}}+\frac{4h^7}{688905}u^{vii}_{n}+\mathcal{O}(h^8),
\end{eqnarray}
\begin{eqnarray}
&&u_{n+3/6}=\overline{u_{n+3/6}}+\frac{h^7}{129024}u^{vii}_{n}+\mathcal{O}(h^8),\\
&&u_{n+4/6}=\overline{u_{n+4/6}}+\frac{h^7}{196830}u^{vii}_{n}+O(h^8),\\
&&u_{n+5/6}=\overline{u_{n+5/6}}+\frac{325h^7}{282175488}u^{vii}_{n}+\mathcal{O}(h^8).
\end{eqnarray}
Also, we have
\begin{eqnarray}
\nonumber u_{n+1}&=&u_{n}+\frac{h}{840}[ 41f_{n}+216f_{n+1/6}+27f_{n+2/6}+272f_{n+3/6}+27f_{n+4/6}\\
&&+216f_{n+5/6}+41f_{n+1}]-\frac{h^9u^{(9)}}{1567641600}.\end{eqnarray}
From all of the above, we deduce that 
\begin{eqnarray}
\nonumber u_{n+1}&=&u_{n}+\frac{h}{840}( 41f_{n}+216\overline{f_{n+1/6}}+27\overline{f_{n+2/6}}+272\overline{f_{n+3/6}}+27\overline{f_{n+4/6}}\\
&&+216\overline{f_{n+5/6}}+41f_{n+1})+t_{n}(h),
\end{eqnarray}
where
\begin{equation*}
t_{n}(h)=\mathcal{O}(h^8).
\end{equation*}
Thus the scheme \eqref{lb4} is seventh order convergent.

\subsection{Stability of the formula \eqref{lb4}}
Consider the test problem
\begin{equation}
u'(t)=-\lambda u(t), \hspace{.5cm}\lambda>0
\end{equation}
and assume $s= h\lambda$, then we have
\begin{equation}
u_{n+1}=\varPsi\left( s\right) u_{n},
\end{equation}
where
\begin{equation*}
\varPsi\left( s\right)=\frac{540\left( 840-414s+82s^2-7s^3\right) }{453600+230040s+48600s^2+5480s^3+540s^4+135s^5+27s^6}.
\end{equation*}
\begin{figure}
	\begin{center}
\includegraphics[height=2in,width=3in]{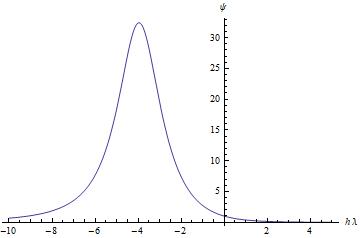}
\caption{Root of characteristic equation }
\label{fg1}
\end{center}
\end{figure}
From  figure \ref{fg1} it can be seen that $\varPsi(s)\nless 1$ and hence our scheme is not $A$-stable. Since $\varPsi(s)\rightarrow 0$ as $s\rightarrow \infty$ and hence scheme is weakly $L$-stable.
 \subsection{Stability Region for the formula \eqref{lb4}}
 We use the boundary locus method \cite[p.64, chapter 7]{LeVeque2007} and determine the boundary of the region. It can be easily seen that outside of the region it is unconditionally stable.
 \begin{figure}[H]
 	\begin{center}
 		\includegraphics[width=3in]{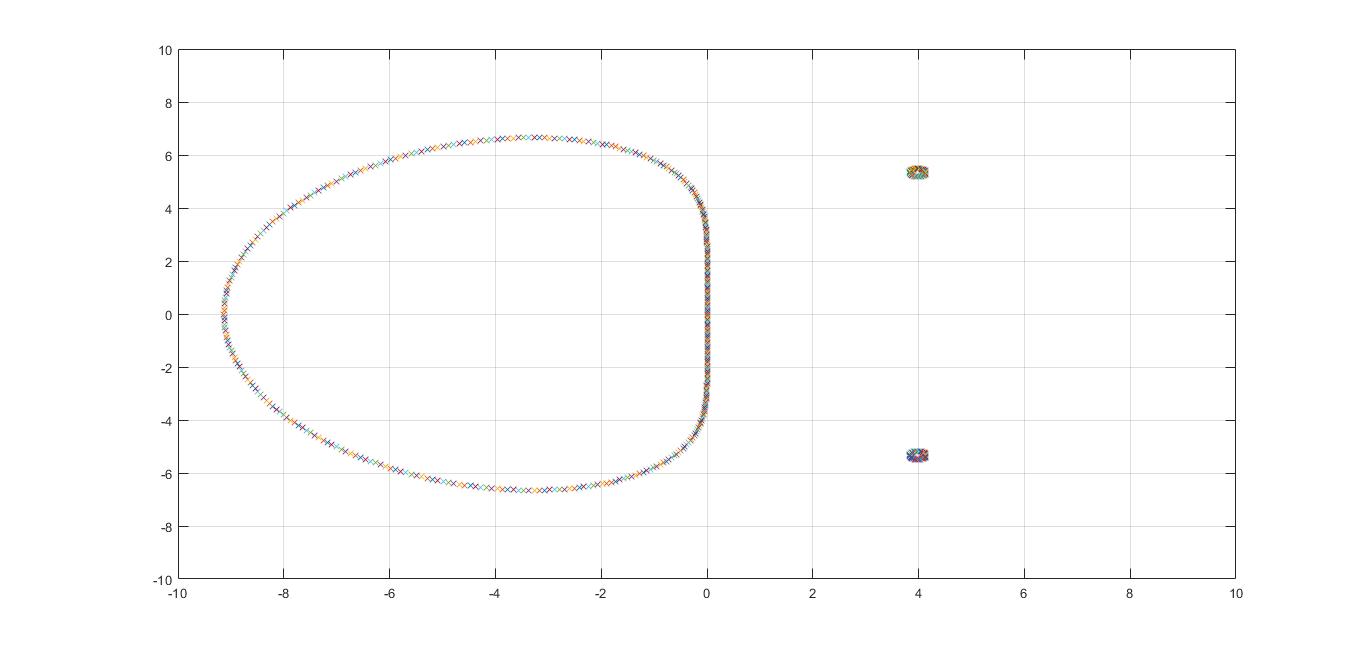}\\
 		\caption{Region of Stability }\label{stabilityfig}
 	\end{center}
 \end{figure}
 \section{Solution of the Burgers' equation}
\subsection{The final scheme}
We discretize the solution space with uniform mesh expressed as ${\Sigma_{T}}_{i,j}=\{(x_{i},t_{j}):i=0,1,2,...,N,j=0,1,2,...,M\}.$  For that, we partition the interval $\left[\alpha_{0} ,\alpha_{1} \right] $ in to $N$ equal sub intervals with the spatial grid $x_{i}=i h,i=0,1,2,...,N$, where $N$ is a positive integer and $h$ is the spatial step.

Also partition the interval $\left[ 0,T \right] $ in to $M$ equal subintervals with the temporal grid $t_{j}=j\tau,j=0,1,2,...,M$ where $\tau=T/M$ and $M$ is a positive integer.

Now define $\psi_{i}(t)=\psi(x_{i},t)$ and consider linearized Burger's equation \eqref{lb2} and compute the solution $\psi(x_{i},t)$ for a given $t$ and for $x_{i}$ on $\left[\alpha_{0} ,\alpha_{1} \right] $. Then we use \eqref{r1}-\eqref{r2} to deduce the following formula for computing the $w(x_i,t_j)$ which is the solution of the nonlinear Burgers' equation \eqref{lb1},
\begin{eqnarray*}
w(x_{i},t)=	\Big(\frac{-\nu_{d}}{2h}\Big)\frac{\psi(x_{i}+h,t)-\psi(x_{i}-h,t)}{\psi(x_{i},t)}.
\end{eqnarray*}
\par Here we approximate second order spatial derivative by fourth order finite difference ratio which is given by
\begin{eqnarray*}
	\frac{\partial^2\psi(x,t)}{\partial^2x}\approx\frac{16(\psi(x+h,t)+\psi(x-h,t))-30\psi(x,t)-(\psi(x+2h,t)+\psi(x-2h,t))}{12h^2},
\end{eqnarray*}
and convert the linearized Burgers' equation into an initial value problem in vector form. 
\par Now, we apply finite difference discretization on \eqref{lb2} with the Neumann boundary conditions
\begin{equation*}
\psi_{x}(\alpha_{i},t)=0,~~i=0,1,
\end{equation*}
we get the following equation

\begin{equation}
\label{lb5}
\frac{\partial\Psi(t)}{\partial t}=-\frac{\nu_{d}}{24h^2} D \Psi(t),
\end{equation}
where $\Psi(t)=\left[\psi_{0}(t),\psi_{1}(t),\psi_{2}(t),...,\psi_{N}(t)\right]^T$ and $D$ is the $(N+1)\times(N+1)$ pentadiagonal matrix given by
\begin{equation}
D=
\begin{pmatrix}
30&-32&2&0&0& \cdots &0&0\\
-16&31&-16&1&0& \cdots &0&0\\
1&-16&30&-16&1&0& \cdots &0\\
\vdots & \vdots & \ddots & \ddots & \ddots & \ddots &  \vdots\\
0&0&0&1&-16&30&-16&1\\
0&0&0&0&1&-16&31&-16\\
0&0&0&0&0&2&-32&30
\end{pmatrix}
\end{equation}
and $\Psi(0)=\left[g(x_{0}),g(x_{1}),g(x_{2}),...,g(x_{N})\right].$\\
Let $\rho=\nu_{d}\tau/24 h^2$, then applying the time integration formula on the initial value problem \eqref{lb5}, we get
\begin{eqnarray}
	\label{lb6}
\nonumber	\Psi_{j+1}&=&\Psi_{j}-\frac{\rho}{840}D\Big(41\Psi_{j}+216\overline {\Psi_{j+1/6}}+27\overline {\Psi_{j+2/6}}+272\overline {\Psi_{j+3/6}}\\
	&& +27\overline {\Psi_{j+4/6}}+216\overline {\Psi_{j+5/6}}+41 \Psi_{j+1}\Big),
\end{eqnarray}
where
\begin{eqnarray}
\nonumber	\overline {\Psi_{j+1/6}}&=&\frac{1}{46656}\Big[ \left( 44875I-6750\rho D+375\rho^2 D^2\right)\Psi_{j}+\Big( 1781I+755 \rho D\\
	&&+\frac{275}{2}\rho^2 D^2+125\big(\frac{\rho^3 D^3}{6}+\frac{\rho^4 D^4}{24}+\frac{\rho^5 D^5}{120}\big)\Big)\Psi_{j+1}\Big],\\
\nonumber	\overline {\Psi_{j+2/6}}&=&\frac{1}{729}\Big[\left( 568I-144\rho D+12\rho^2 D^2\right)\Psi_{j}+\Big( 161I+62 \rho D+10\rho^2 D^2\\
	&&+8\big( \frac{\rho^3 D^3}{6}+\frac{\rho^4 D^4}{24}+\frac{\rho^5 D^5}{120}\big)\Big)\Psi_{j+1}\Big],\\
\nonumber	\overline{\Psi_{j+3/6}}&=&\frac{1}{64}\Big[ \left( 31I-10\rho D+\rho^2 D^2\right)\Psi_{j}+\Big( 33I+11 \rho D+\frac{3}{2}\rho^2 D^2+\big(\frac{\rho^3 D^3}{6}\\
	&&+\frac{\rho^4 D^4}{24}+\frac{\rho^5 D^5}{120}\big)\Big)\Psi_{j+1}\Big],\\
\nonumber	\overline {\Psi_{j+4/6}}&=&\frac{1}{729}\Big[ \left( 145I-54\rho D+6\rho^2 D^2\right)\Psi_{j}+\Big( 584I+152 \rho D+20\rho^2 D^2\\
	&&+8 \big(\frac{\rho^3 D^3}{6}+\frac{\rho^4 D^4}{24}+\frac{\rho^5 D^5}{120}\big)\Big)\Psi_{j+1}\Big],\\
\nonumber	\overline {\Psi_{j+5/6}}&=&\frac{1}{46656}\Big[ \left( 1531I-630\rho D+75\rho^2 D^2\right)\Psi_{j}+\Big( 45125I+ 6875 \rho D \\
	&&+\frac{875}{2}\rho^2 D^2 +125\big(\frac{\rho^3 D^3}{6}+\frac{\rho^4 D^4}{24}+\frac{\rho^5 D^5}{120}\big)\Big)\Psi_{j+1}\Big].
\end{eqnarray}
Now, use $\overline {\Psi_{j+1/6}},~~\overline {\Psi_{j+2/6}},~~\overline {\Psi_{j+3/6}},~~\overline {\Psi_{j+4/6}},~~\overline {\Psi_{j+5/6}}$  in Equation \eqref{lb6}, we get
\begin{eqnarray}
\nonumber&&(453600I+2300\rho D+48600 \rho^2D^2+5480 \rho^3D^3+540 \rho^4D^4+135 \rho^5D^5+27\rho^6D^6)~\overline{\Psi_{j+1}}\\
\label{l10}&&~~~~~~~=540(840I-414\rho D+84\rho^2D^2-7\rho^3 D^3)~\overline{\Psi_{j}}.
\end{eqnarray}
This method is $\mathcal{O}(h^4)+\mathcal{O}(\tau^7).$ BY using \eqref{lb6} we can compute $\Psi_{j+1}$ and hence $w_{ij}$ is computed at different $x_i$'s for a given time level $t_j$. Physical properties of the solutions are discussed later in form of figures and tables. 

\subsection{Unconditional stability of the scheme for the heat equation}
 Equation \eqref{lb6} can be written as
\begin{equation*}
\Psi_{j+1}=P \Psi_{j},
\end{equation*}
where
\begin{eqnarray}
\nonumber P&=&\frac{540(840I-414\rho D+84\rho^2D^2-7\rho^3 D^3)}{(453600I+230040\rho D+48600 \rho^2D^2+5480 \rho^3D^3+540 \rho^4D^4+135 \rho^5D^5+27\rho^6D^6)}\\
&=&L_{1}^{-1}L_{2},~\text{(say)},
\end{eqnarray}
where
\begin{eqnarray}
\nonumber&&L_{1}=(453600I+230040\rho D+48600 \rho^2D^2+5480 \rho^3D^3+540 \rho^4D^4+135 \rho^5D^5+27\rho^6D^6)\\
\nonumber&&L_{2}=540(840I-414\rho D+84\rho^2D^2-7\rho^3 D^3).
\end{eqnarray}
\begin{lemma}
	The matrix $P$ is similar to a symmetric matrix.
\end{lemma}
\begin{proof}
	Let us introduce a diagonal matrix
	\begin{equation*}
	Q=
	\begin{pmatrix}
	\sqrt{2}&&&& &\\
	&1&&&&&\\
	&&\ddots&&&&\\
	&&&1&&&\\
	&&&&&\sqrt{2}
	\end{pmatrix}
	\end{equation*}
	such that 
	\begin{equation*}
	\tilde{D}=Q^{-1}DQ,
	\end{equation*}
	i.e., $D$ is similar to a symmetric matrix $\tilde{D}$.\\
	Now, we will show that $P$ is similar to symmetric matrix. Let
	\begin{equation*}
	\begin{split}
	\tilde{P}=Q^{-1}PQ=Q^{-1}L_{1}^{-1}L_{2}Q=[Q^{-1}L_{1}^{-1}Q][Q^{-1}L_{2}Q]\\
	=[Q^{-1}L_{1}Q]^{-1}[Q^{-1}L_{2}Q]=\tilde{L_{1}^{-1}}\tilde{L_{2}}.
	\end{split}
	\end{equation*}
	\begin{eqnarray*}
		\tilde{L}_{1}&=&(453600I+230040\rho \tilde{D}+48600 \rho^2\tilde{D}^2+5480 \rho^3\tilde{D}^3+540 \rho^4\tilde{D}^4+135 \rho^5\tilde{D}^5+27\rho^6\tilde{D}^6)\\
		\tilde{L}_{2}&=&540(840I-414\rho \tilde{D}+84\rho^2\tilde{D}^2-7\rho^3 \tilde{D}^3)
	\end{eqnarray*}
	but matrices $\tilde{L}_{1}^{-1}$  and $\tilde{L}_{2}$ are symmetric and commute and therefore
	$P$ is similar to a symmetric matrix $\tilde{P}$ and therefore all the eigen values of the matrix $P$ are real.
	\end{proof}
\begin{lemma}
	Eigen values of the matrix $D$ are $30+2 \cos(2 l\pi /N)-32 \cos(l \pi/N),~l=0(1)N.$
\end{lemma}
\begin{proof}
	Let $V=\{v_{1},v_{2},v_{3},...,v_{N+1}\}$ be the eigen vectors of the matrix $D$ corresponding to the eigen value $\lambda_{l}.$ Then, we have
	\begin{equation*}
	(30-\lambda_{l})v_{1}-32v_{2}+2v_{3}=0,
	\end{equation*}
	\begin{equation*}
	-16v_{1}+(31-\lambda_{l})v_{2}-16v_{3}+v_{4}=0,
		\end{equation*}
		\begin{equation*}
	v_{j-2}-16v_{j-1}+(30-\lambda_{l})v_{j}-16v_{j+1}+v_{j+2}=0,~~j=2,3,4,...,N-1,
		\end{equation*}
			\begin{equation*}
		v_{N-2}-16v_{N-1}+(31-\lambda_{l})v_{N}-16v_{N+1}=0,
		\end{equation*}
		\begin{equation*}
	2v_{N-1}-32v_{N}+(30-\lambda_{l})v_{N+1}=0.
		\end{equation*}
	We set $v_{1}=v_{3},~v_{0}=v_{2},~v_{N}=v_{N+2},~v_{N-1}=v_{N+3}$ then we get fourth order difference equation 
		\begin{equation}
		\label{c1}
	v_{j-2}-16v_{j-1}+(30-\lambda_{l})v_{j}-16v_{j+1}+v_{j+2}=0,~~j=1,2,...,N+1,
	\end{equation}
	with the boundary conditions  $v_{1}=v_{3},~v_{0}=v_{2},~v_{N}=v_{N+2},~v_{N-1}=v_{N+3}$.
The characteristc equation of the equation \eqref{c1} is $m^4-16m^3+(30-\lambda_{l})m^2-16m+1=0$. Assume $m_{1},m_{2},m_{3},m_{4}$ are the characteristc roots then we have
\begin{equation*}
m_{1}+m_{2}+m_{3}+m_{4}=16,
\end{equation*}
\begin{equation*}
m_{1}m_{2}+m_{1}m_{3}+m_{1}m_{4}+m_{2}m_{3}+m_{2}m_{4}+m_{3}m_{4}=(30-\lambda_{l}),
\end{equation*}
\begin{equation*}
m_{1}m_{2}m_{3}+m_{1}m_{3}m_{4}+m_{2}m_{3}m_{4}+{m_{1}m_{2}m_{4}}=16,
\end{equation*}
\begin{equation*}
m_{1}m_{2}m_{3}m_{4}=1,
\end{equation*}
and the solution is given by $v_{j}=C_{1}m_{1}^j+C_{2}m_{2}^j+C_{3}m_{3}^j+C_{4}m_{4}^j$. Let $m_{1}=re^{i \theta}$ then setting $r=1$, gives $\lambda_{l}=30+2 \cos2\theta-32 \cos\theta$. Also, it can be shown that for all possible values of $m_{2},m_{3},m_{4}$ $\theta=l \pi/N,~l=0,1,2,...,N.$ Since $V$ is the non trivial vector satisfying $DV=\lambda_{l}V$, therefore the eigenvalues of $D$ are $\lambda_{l}=30+2 \cos(2 l\pi /N)-32 \cos(l \pi/N),~l=0(1)N.$ Also, it can be shown that $\lambda_{l}\geq 0,~~\forall~ l.$
\end{proof}
Now let $\varrho (P)$ be the spectral radius of the matrix $P$ then 
\begin{equation*}
\varrho (P)=\varrho (\bar P),
\end{equation*}
and is given by 
\begin{equation*}
\varrho (P)\leq max_{l}| \mu_{l}|,
\end{equation*}
where $\mu_{l}(l=0,1,2,...N)$ are the eigen values  of the matrix $L_{1}^{-1}L_{2}$ and therefore 
\begin{eqnarray}
\nonumber \mu_{l}&=&\frac{540(840I-414\rho \lambda_{l}+84\rho^2\lambda_{l}^2-7\rho^3 \lambda_{l}^3)}{(453600I+230040\rho \lambda_{l}+48600 \rho^2\lambda_{l}^2+5480 \rho^3\lambda_{l}^3+540 \rho^4\lambda_{l}^4+135 \rho^5\lambda_{l}^5+27\rho^6\lambda_{l}^6)},\\
\nonumber l=0,1,2,...,N.
\end{eqnarray}
It is clear that the eigen value $\mu_{l}\leq 1$ for all possible values of $ \rho > 0$ and hence the method is unconditionally stable. Using Taylor's series, it can also be shown that the method is consistent with the differential equation.

\section{Numerical Experiment}
To validate the accuracy and efficiency of the scheme developed in this paper, here we consider six example with Dirichlet BCs and different ICs. We compute error between the proposed scheme and the analytical solution which is measured by the mean root square error norm $L_{2}$ and maximum error norm $L_{\infty}$ defined by
\begin{eqnarray}
	\label{e1}
	L_{2}=\sqrt{h\sum_{j=0}^{N}\mid w^{exact}_{j}-(w_{Nu})_{j} \mid^{2}},~~~L_{\infty}=\Arrowvert w^{exact}-w_{Nu} \Arrowvert _{\infty}= \mathop{max}_{j}\mid w^{exact}_{j}-(w_{Nu})_{j} \mid,
\end{eqnarray}
where $ w^{exact}_{j}$ is the exact solution and $(w_{Nu})_{j}$ is numerical solution at the $jth$ spatial point.
  We also consider problem \ref{l15} and \ref{l16} in which ICs are inconsistent with the BCs. We depict the analytical solution and numerical solution with the help of table and figures. We use Mathematica 11.3 to compute the solution and the results are plotted with the help of the software Origin $8.5$. 
\subsection{Example 1}
\label{p1}
Consider the equation \eqref{lb1} with Dirichlet BCs
\begin{equation}
w(\alpha_{i},t)=0,~~ \alpha_{i}=i,~~i=0,1 ~~\mbox{and} ~~t\in(0,T],
\end{equation}
and ICs
\begin{equation}
w(x,0)=\sin(\pi x), \hspace{.5cm} x\in(0,1),
\end{equation}
where $\nu_{d}$ is the coefficient of viscosity. 
Using the transformation
\begin{equation}
w(x,t)=\frac{-\nu_{d}\psi_{x}}{\psi},
\end{equation}
the equation \eqref{lb1} will be
\begin{equation}
\psi_{t}=\frac{\nu_{d}}{2}\psi_{xx},\hspace{.5cm} x\in(0,1),t>0,
\end{equation}
with ICs and BCs
\begin{eqnarray}
	&&\psi(x,0)=exp\Big(\frac{1}{\pi \nu_{d}}(\cos\pi x-1)\Big),\\
	&&\psi_{x}(\alpha_{i},t)=0,~~ \alpha_{i}=i,~~ i=0,1 ~~\mbox{and}~~ t>0.
	\end{eqnarray}
	Now, equation \eqref{ANBE} represents the analytical solution  where
	\begin{eqnarray}
		&&\beta_{0}=\int_{0}^1 \exp\Big(\frac{1}{\pi \nu_{d}}(\cos\pi x-1)\Big)dx,\\
		&&\beta_{l}=2\int_{0}^1 \exp \Big(\frac{1}{\pi \nu_{d}}(\cos\pi x-1)\cos l\pi x\Big)dx,
	\end{eqnarray}
	are Fourier coefficients.\par
	In this example, in order to measure the accuracy of the numerical solution, the discrete $L_{2}$ and $L_{\infty}$-error defined in the equation \eqref{e1} are used to compute the difference between the analytical solution and numerical solution at different specified time. In table \ref{tb1}, we take $\nu_{d}=2$ and time step $\tau=0.0001$ with $h=0.0125$. It is observed that the computed results are in good agreement with the analytical solution. It can be seen that as time passes the numerical solution decreases at the same location. While with the location changing from $0$ to $1$, the solution first increases and then decreases at the same moment.  Table \ref{tb2} lists the numerical solution and analytical solution obtained by the present method for $\nu_{d}=0.2,~h=0.0125$ and with the time step $\tau=0.0001$. The obtained result are compared with the existing result in the literature and can be seen  that the present method gives better result than the results  obtained by the method in \cite{kutluay2004numerical} and \cite{asaithambi2010numerical}.  From table \ref{tb3}, it can be seen that the present scheme gives satisfactory result for small values of viscosity coefficient $\nu_{d}=0.01$  at different time for $\tau=0.01$ and $h=0.0125$.\par 
	Figure \ref{fg3} depicts the accuracy of the numerical solution at different time for $\nu_{d}=0.2$ and we are not able to distinguish between the exact solution and numerical solution. It is known that the Fourier series solution fails to converge for $\nu_{d}<0.01$. This is because  the rate of convergence of infinite series \eqref{ANBE} is slow for the small value of $\nu_{d}$. From figure \ref{fg4}, it can be seen that the analytical solution shows high oscillation while the numerical solution obtained by present method follows the physical behaviour.  Figure \ref{fg5} represents the physical behaviour of the numerical solution for the different small values of $\nu_{d}$. In figure \ref{fg6} we illustrate the physical behaviour of the computed solutions with the help of  three-dimensional figures for small value of $\nu_{d}$.
\begin{table}[!ht]
	\caption{Solution of problem \ref{p1} for $h=0.0125$ at different value $T$ for $\nu_{d}=2$ and $\tau=0.0001$}
	\label{tb1}										
	\centering											
	\begin{center}											
		\resizebox{17cm}{3cm}{
			\begin{tabular}{lllllll}
				\hline
				$x$\hspace{1cm} & \multicolumn{2}{l}{T=0.001}        & \multicolumn{2}{l}{T=0.01}         & \multicolumn{2}{l}{T=0.1}          \\ 
				& \multicolumn{2}{l}{\noindent\rule{6.2cm}{0.4pt}}        & \multicolumn{2}{l}{\noindent\rule{6.2cm}{0.4pt}}         & \multicolumn{2}{l}{\noindent\rule{6.2cm}{0.4pt}}\\
				& Computed Solution & Exact Solution & Computed Solution & Exact Solution & Computed Solution & Exact Solution \\ \hline
				0.1 & 0.304976          & 0.305088       & 0.273145          & 0.273239       & 0.109509          & 0.109538       \\
				0.2 & 0.580361          & 0.580565       & 0.521393          & 0.521564       & 0.209737          & 0.209792       \\
				0.3 & 0.799363          & 0.799621       & 0.721630           & 0.721852       & 0.291820           & 0.291896       \\
				0.4 & 0.940545          & 0.940817       & 0.854348          & 0.854590        & 0.347834          & 0.347924       \\
				0.5 & 0.989926          & 0.990174       & 0.905483          & 0.905713       & 0.371482          & 0.371577       \\
				0.6 & 0.942407          & 0.942609       & 0.868137          & 0.868334       & 0.358954          & 0.359046       \\
				0.7 & 0.802375          & 0.802522       & 0.743949          & 0.744098       & 0.309827          & 0.309905       \\
				0.8 & 0.583373          & 0.583466       & 0.543723          & 0.543821       & 0.227760           & 0.227817       \\
				0.9 & 0.306837          & 0.306881       & 0.286951          & 0.286999       & 0.120656          & 0.120687   \\ \hline
				$L_{\infty}~~error$ &2.71275E-4 & &2.413E-04 & & 9.54852E-05\\
				$L_{2}~~error$ &6.41526E-05&   &5.82562E-05&   &2.27535E-05 \\ \hline
		\end{tabular}}
	\end{center}
\end{table}
\begin{table}[!ht]
	\caption{ Comparison of existing, present numerical result and exact solution for $h=0.0125,~\nu_{d}=0.2,~\tau=0.0001$ at different  value of $T$ for problem \ref{p1}}
	\label{tb2} 
	\centering
	\begin{tabular}{llllll}
		\hline
		$x$ \hspace{0.5in} & $T$ \hspace{0.5in} & FEM \cite{kutluay2004numerical}   \hspace{0.5in} &  Asai\cite{asaithambi2010numerical}    \hspace{0.5in} & Present  \hspace{0.5in} & Exact   \\ \hline
		0.25 & 0.4 & 0.31215 & 0.30891 & 0.3087531 & 0.30889 \\
		& 0.6 & 0.24360  & 0.24076 & 0.2406489 & 0.24074 \\
		& 0.8 & 0.19815 & 0.19570  & 0.1956120  & 0.19568 \\
		& 1   & 0.16473 & 0.16259 & 0.1625168 & 0.16256 \\
		& 3   & 0.02771 & 0.02722 & 0.0271953 & 0.02720  \\
		0.50 & 0.4 & 0.57293 & 0.56970  & 0.5694998 & 0.56963 \\
		& 0.6 & 0.45088 & 0.44728 & 0.4470928 & 0.44721 \\
		& 0.8 & 0.36286 & 0.35932 & 0.3591441 & 0.35924 \\
		& 1   & 0.29532 & 0.29200   & 0.2918410  & 0.29192 \\
		& 3   & 0.04097 & 0.04023 & 0.0401946 & 0.04021 \\
		0.75 & 0.4 & 0.63038 & 0.62567 & 0.6254715 & 0.62544 \\
		& 0.6 & 0.49268 & 0.48747 & 0.4871652 & 0.48721 \\
		& 0.8 & 0.37912 & 0.37415 & 0.3738557 & 0.37392 \\
		& 1   & 0.29204 & 0.28766 & 0.2874128 & 0.28747 \\
		& 3   & 0.03038 & 0.02979 & 0.0297645 & 0.02977\\ \hline
	\end{tabular}
\end{table}

\begin{table}[!ht]
	\begin{center}
		\caption{Solution of problem \ref{p1}  at different value of $T$ for $\nu_{d}=0.01,~h=0.0125$ and $\tau=0.01$}
		\label{tb3}	
		\begin{center}
			\resizebox{7cm}{3.5cm}{
				\begin{tabular}{llll}
					\hline
					$x$  & $T$ & Numerical Solution & Exact Solution \\ \hline
					0.25 & 5   & 0.046922           & 0.046963       \\
					& 10  & 0.024202           & 0.024217       \\
					& 15  & 0.016300             & 0.016308       \\
					& 20  & 0.012236           & 0.012240        \\
					0.50 & 5   & 0.093998           & 0.093920        \\
					& 10  & 0.048414           & 0.048421       \\
					& 15  & 0.032431           & 0.032439       \\
					& 20  & 0.023883           & 0.023889       \\
					0.75 & 5   & 0.141354           & 0.140832       \\
					& 10  & 0.071175           & 0.071134       \\
					& 15  & 0.044135           & 0.044133       \\
					& 20  & 0.029155           & 0.029159     \\ \hline 
			\end{tabular}}
		\end{center}
	\end{center}
\end{table}

\begin{figure}[h!]
	\begin{center}
		\includegraphics[width=3in]{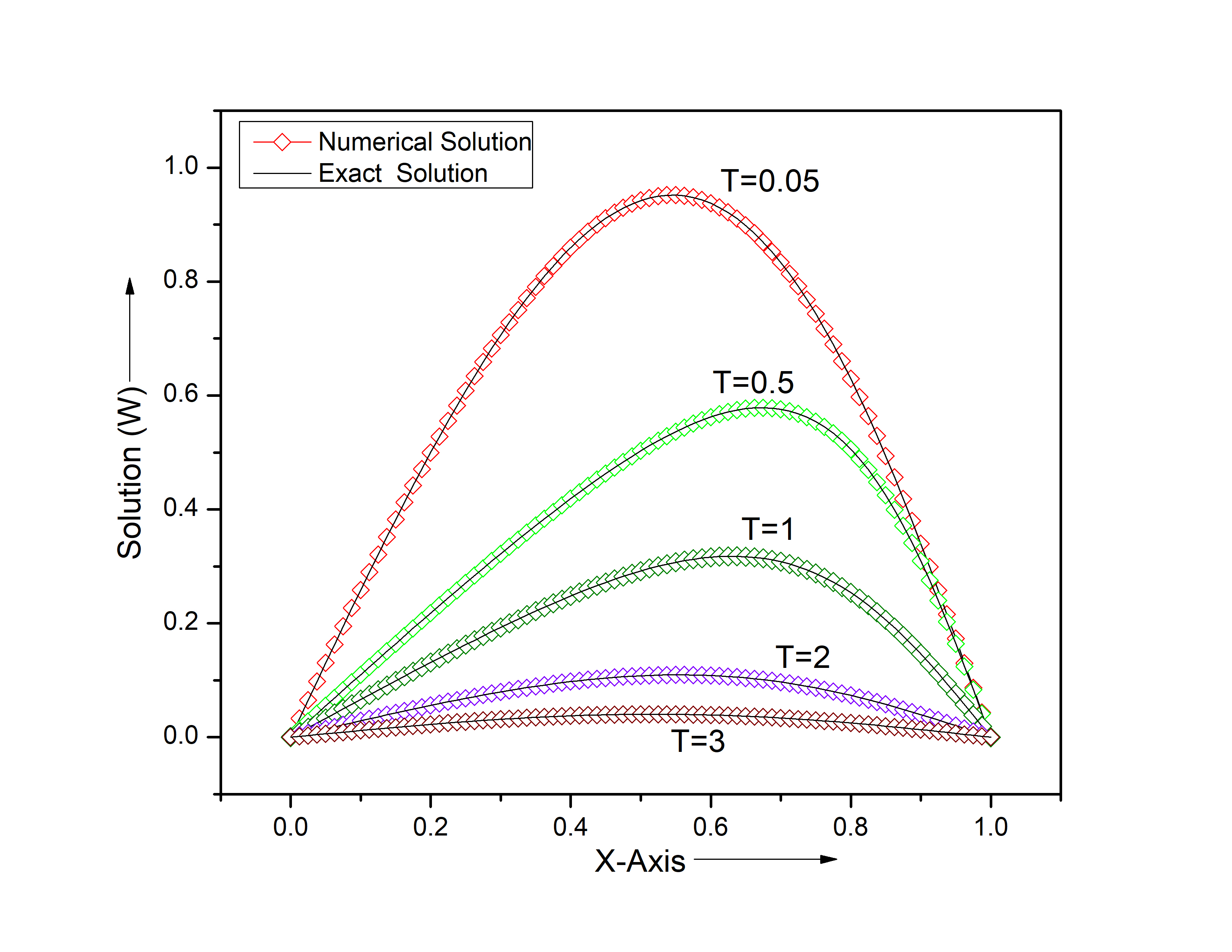}
		\caption{Comparison of Numerical and analytical solution for problem \ref{p1} at different time $T$, $\nu_{d}=0.2,~\tau=0.001$ and $h=0.0125$.}
		\label{fg3}
	\end{center}
\end{figure}

\begin{figure}[h!]
	\begin{center}
		\includegraphics[width=3in]{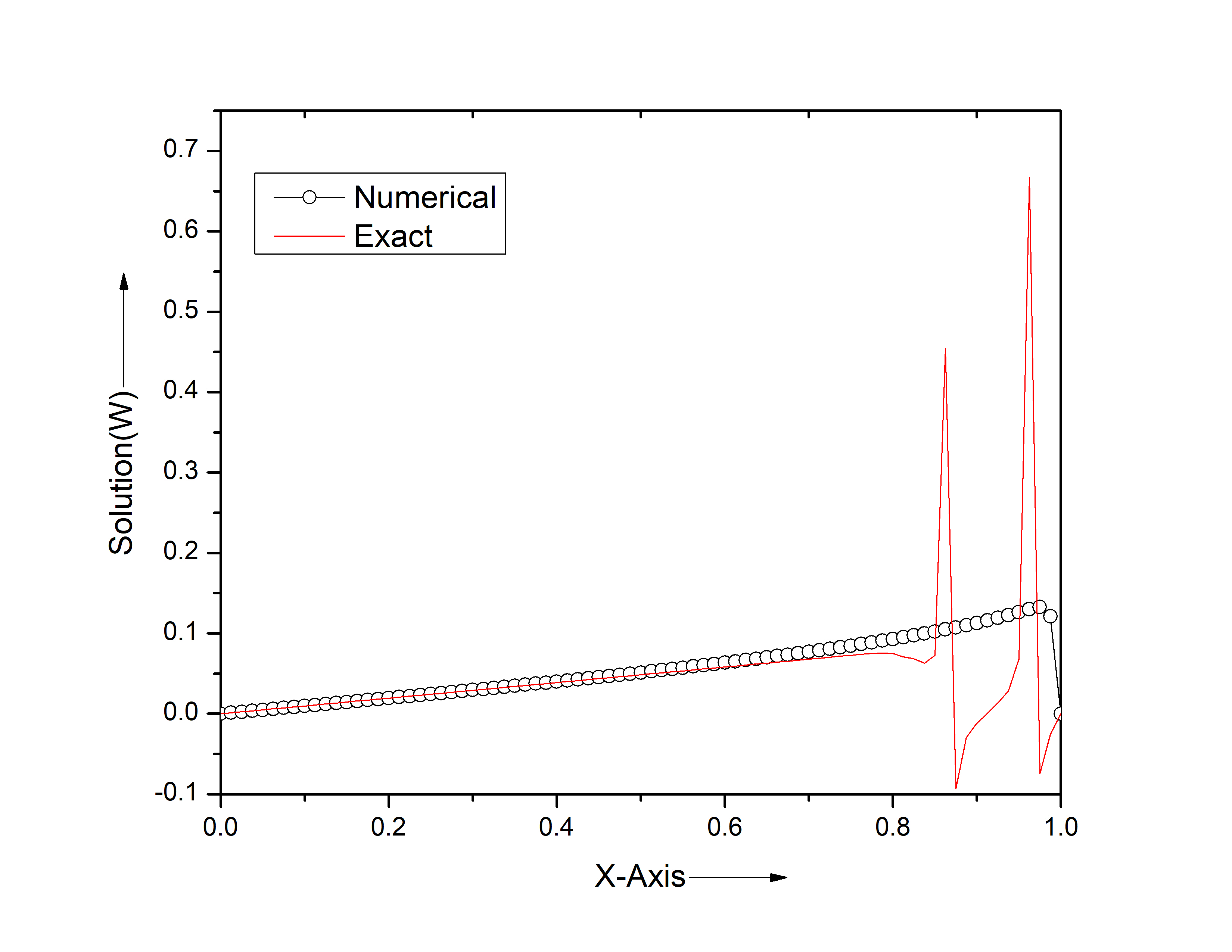}
		\caption{Comparison of Numerical and analytical solution for Problem \ref{p1} at time $T=10$, $\nu_{d}=0.001,\tau=0.001$ and $h=0.0125$.}
		\label{fg4}
	\end{center}
\end{figure}

\begin{figure}[h!]
	\begin{center}
		\includegraphics[width=3in]{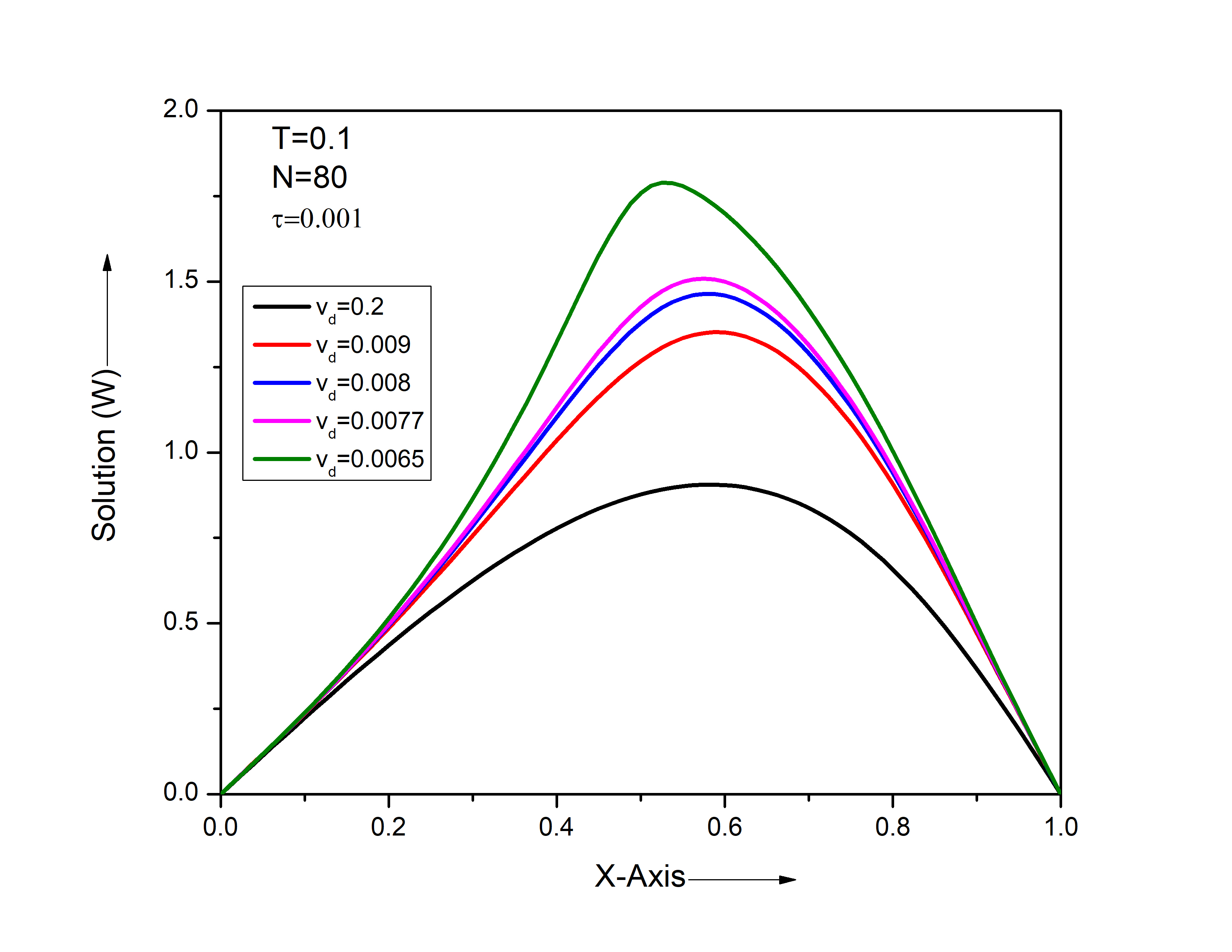}
		\caption{Numerical Solution of problem \ref{p1}  at time $T=0.1$, for different small values of $\nu_{d}$ with $\tau=0.001$ and $h=0.0125$.}	
		\label{fg5}
	\end{center}
\end{figure}

\begin{figure}[h!]
	\begin{center}
		\includegraphics[width=3in]{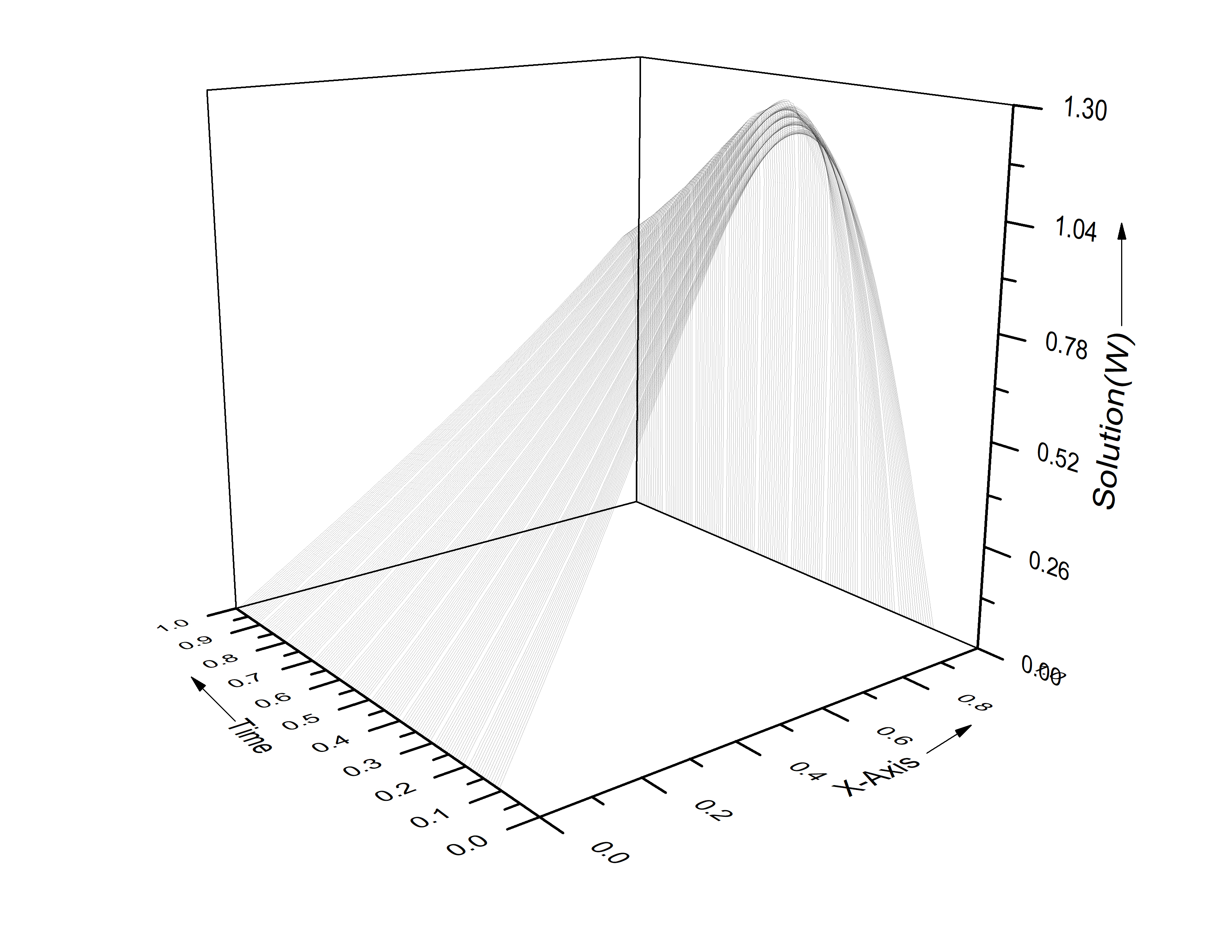}
		\caption{Numerical Solution of problem \ref{p1} at different time $T$, $\nu_{d}=0.01$, $h=0.0125$ and $\tau=0.001$. }
		\label{fg6}
	\end{center}
\end{figure}

\subsection{Example 2}
\label{p2}
Consider the equation \eqref{lb1} with Dirichlet BCs
\begin{equation}
w(\alpha_{i},t)=0,~~ \alpha_{i}=i,~~i=0,1 ~~\mbox{and} ~~t\in(0,T],
\end{equation}
and ICs
\begin{equation}
w(x,0)=4x(1-x), \hspace{.5cm} x\in(0,1),
\end{equation}
where $\nu_{d}$ is the coefficient of viscosity. 
Using the transformation
\begin{equation}
w(x,t)=\frac{-\nu_{d}\psi_{x}}{\psi},
\end{equation}
we see that the equation \eqref{ANBE} represents the analytic solution where
\begin{eqnarray}
	&&\nonumber \beta_{0}=\int_{0}^1 \exp \Big(-\frac{2x^2}{3 \nu_{d}}(3-2 x\Big)dx,\\
	&&\nonumber \beta_{l}=2\int_{0}^1 \exp \Big(-\frac{2 x^2}{3 \nu_{d}}(3-2x)\cos l\pi x\Big)dx.
\end{eqnarray}
In table \ref{tb4}, we depict the numerical result and compare with the exact solution at different sptial point for $\tau=0.0001$, $\nu_{d}=2$ with  $h=0.0125$  and observed that the numerical result are very closed to the exact solution. $L_{\infty}$ and $L_{2}$ -error indicate that the difference between analytical solution and numerical solution are very less. For the comparison purpose, in table \ref{tb5}, we  take $\nu_{d}=0.2,~\tau=0.0001$ and $h=0.0125$ and observed that the present method gives slightly better result compare to the numerical result in \cite{kutluay2004numerical} and \cite{asaithambi2010numerical}. It can also be seen that when time goes on, the solution get decreases at the same location. Also while location changes from $0$ to $1$, first solution increases and then decreases at each moment and therefore it follows the parabolic profile. In table \ref{tb6}, we depict the numerical solution and analytical solution at different time for small value of $\nu_{d}=0.01$ by taking $\tau=0.01$ and $h=0.0125$. It can be seen that the numerical results are very closed to the exact solution.\par
In figure \ref{fg9}, it can be seen that the analytical solution start oscillation between $x=0.8$ to $x=1$ for small value of $\nu_{d}=0.001$ because the slow convergence of the infinite series  for small value of $\nu_{d}$ but solution obtained by this method follows the parabolic profile. Figure \ref{fg7} demonstrate the accuracy of the method for $\nu_{d}=0.1$ and it can be seen that analytical solution and numerical solution are almost same at different time throughout the domain. Figure \ref{fg10} shows that result obtained by the present method  follows the nature of the solution for different small values of $\nu_{d}$. Figure \ref{fg8} illustrate the physical nature of the solution in three dimension. 
\begin{table}[!ht]
	\caption{Solution of problem \ref{p2} for $h=0.0125$ at different value $T$ for $\nu_{d}=2$ and $\tau=0.0001$}
	\label{tb4}	
	\centering
	\begin{center}	
		\resizebox{16cm}{3.5cm}{
			\begin{tabular}{lllllll}
				\hline
				$x$\hspace{1cm} & \multicolumn{2}{l}{T=0.001}        & \multicolumn{2}{l}{T=0.01}         & \multicolumn{2}{l}{T=0.1}          \\ 
				& \multicolumn{2}{l}{\noindent\rule{6.2cm}{0.4pt}}        & \multicolumn{2}{l}{\noindent\rule{6.2cm}{0.4pt}}         & \multicolumn{2}{l}{\noindent\rule{6.2cm}{0.4pt}}\\
				& Computed Solution & Exact Solution & Computed Solution & Exact Solution & Computed Solution & Exact Solution \\ \hline
				0.1 & 0.350702990         & 0.350947       & 0.294821969        & 0.294953       & 0.112863           & 0.112892       \\
				0.2 & 0.630240123        & 0.630504       & 0.552873368        & 0.553085       & 0.216195           & 0.216252       \\
				0.3 & 0.830425346        & 0.830681       & 0.749515568        & 0.749751       & 0.300887           & 0.300966       \\
				0.4 & 0.951009637        & 0.951242       & 0.873232122        & 0.873459       & 0.358770            & 0.358863       \\
				0.5 & 0.991793845        & 0.991996       & 0.919517990         & 0.919723       & 0.383324           & 0.383422       \\
				0.6 & 0.952578533        & 0.952752       & 0.886057211        & 0.886239       & 0.370563           & 0.370658       \\
				0.7 & 0.833164134        & 0.833318       & 0.771302597        & 0.771464       & 0.319985           & 0.320066       \\
				0.8 & 0.633350801        & 0.633500         & 0.576137870         & 0.576273       & 0.235312           & 0.235371       \\
				0.9 & 0.352988009        & 0.353149       & 0.310053369        & 0.310136       & 0.124687           & 0.124718  \\ \hline   
				$L_{\infty}~~error$ &2.64275E-04 & &2.35909E-04 & & 9.85169E-05\\
				$L_{2}~~error$ &6.55334E-05&   &6.07706E-05&   &2.46429E-05 \\ \hline 
		\end{tabular}}
	\end{center}
\end{table}

\begin{table}[!ht]
	\caption{Comparison of existing, present numerical result and exact solution for $h=0.0125,~\nu_{d}=0.2,~\tau=0.0001$ at different  value of $T$ for problem \ref{p2}.} 
	\label{tb5}
	\centering
	\begin{tabular}{llllll}
		\hline
		$x$ \hspace{0.5in} & $T$ \hspace{0.5in} & FEM\cite{kutluay2004numerical}    \hspace{0.5in} &  Asai \cite{asaithambi2010numerical}    \hspace{0.5in} & Present  \hspace{0.5in} & Exact   \\ \hline
		0.25 & 0.4 & 0.32091 & 0.31754 & 0.317374 & 0.31752 \\
		& 0.6 & 0.24910 & 0.24616 & 0.246045 & 0.24614 \\
		& 0.8 & 0.20211 & 0.19958 & 0.199490  & 0.19956 \\
		& 1   & 0.16782 & 0.16562 & 0.165549 & 0.16560  \\
		& 3   & 0.02828 & 0.02777 & 0.027752 & 0.02776 \\
		0.50 & 0.4 & 0.58788 & 0.58460  & 0.584404 & 0.58458 \\
		& 0.6 & 0.46174 & 0.45805 & 0.457862 & 0.45798 \\
		& 0.8 & 0.37111 & 0.36748 & 0.367304 & 0.36740  \\
		& 1   & 0.30183 & 0.29843 & 0.298267 & 0.29834 \\
		& 3   & 0.04185 & 0.41090  & 0.041054 & 0.04107 \\
		0.75 & 0.4 & 0.65054 & 0.64586 & 0.645660  & 0.64562 \\
		& 0.6 & 0.50825 & 0.50294 & 0.502629 & 0.50268 \\
		& 0.8 & 0.39068 & 0.38557 & 0.385269 & 0.38534 \\
		& 1   & 0.30057 & 0.29605 & 0.295794 & 0.29586 \\
		& 3   & 0.03106 & 0.03046 & 0.030432 & 0.03044 \\ \hline
	\end{tabular}
\end{table}
\begin{table}[!ht]
	\caption{Solution of problem \ref{p2}  at different time $T$ for $\nu_{d}=0.01,~h=0.0125$ and $\tau=0.01$}
	\label{tb6}	
	\begin{center}
		\begin{center}
			\resizebox{6cm}{3cm}{
				\begin{tabular}{llll}
					\hline
					$x$  & $T$ & Numerical Solution & Exact Solution \\ \hline
					0.25 & 5   & 0.047372           & 0.047415       \\
					& 10  & 0.024321           & 0.024336       \\
					& 15  & 0.016355           & 0.016362       \\
					& 20  & 0.012268           & 0.012272       \\
					0.50 & 5   & 0.094895           & 0.094814       \\
					& 10  & 0.048653           & 0.048660        \\
					& 15  & 0.032542           & 0.032550        \\
					& 20  & 0.023951           & 0.023957       \\
					0.75 & 5   & 0.142693           & 0.142154       \\
					& 10  & 0.071560            & 0.071517       \\
					& 15  & 0.044330            & 0.044328       \\
					& 20  & 0.029271           & 0.029275    \\ \hline  
			\end{tabular}}
		\end{center}
	\end{center}
\end{table}

\begin{figure}[h!]
	\begin{center}
		\includegraphics[width=4in]{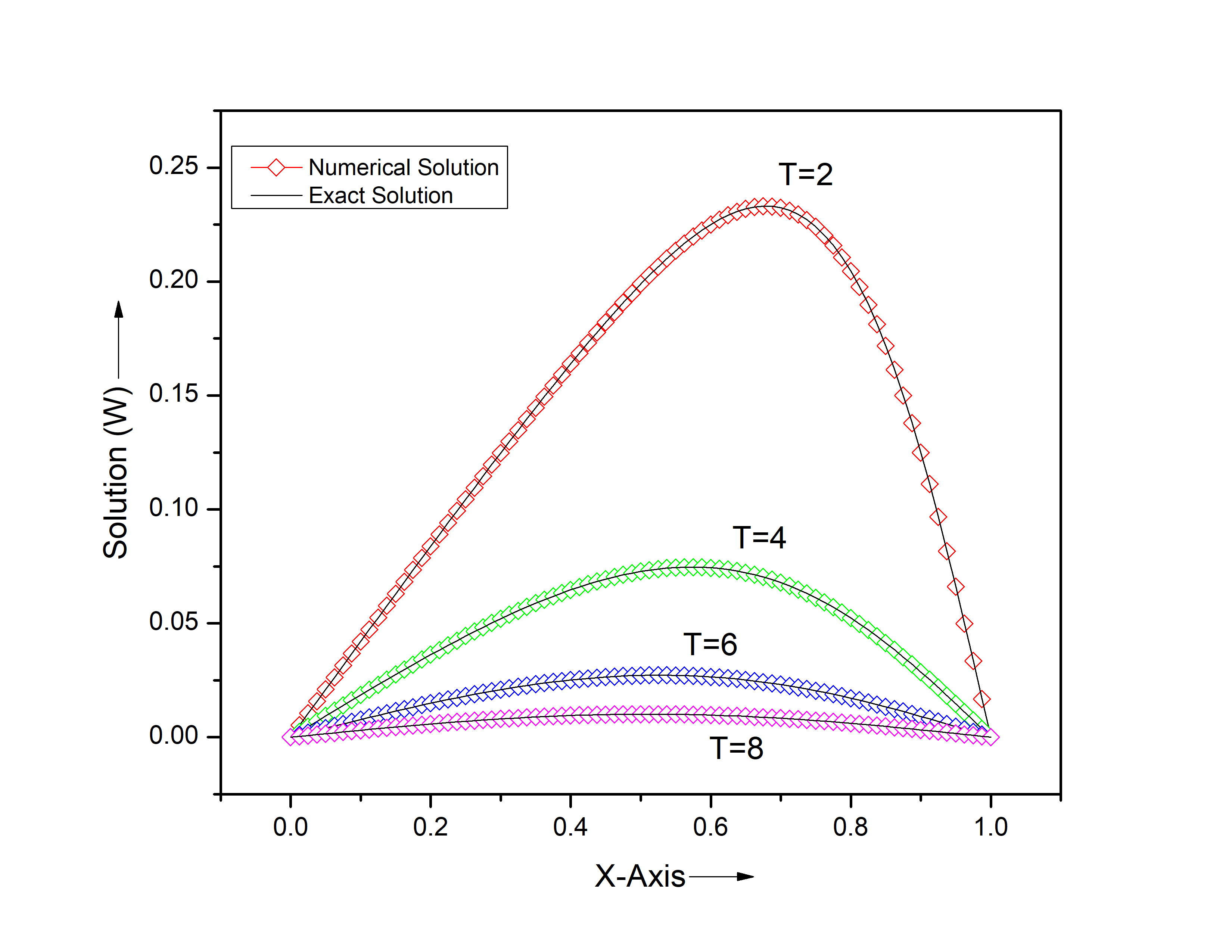}
		\caption{Comparison of Numerical and analytical solution for problem \ref{p2} at different time $T$ , $\nu_{d}=0.1,h=0.0125$ and $\tau=0.001$.}
		\label{fg7}
	\end{center}
\end{figure}

\begin{figure}[h!]
	\begin{center}
		\includegraphics[width=3in]{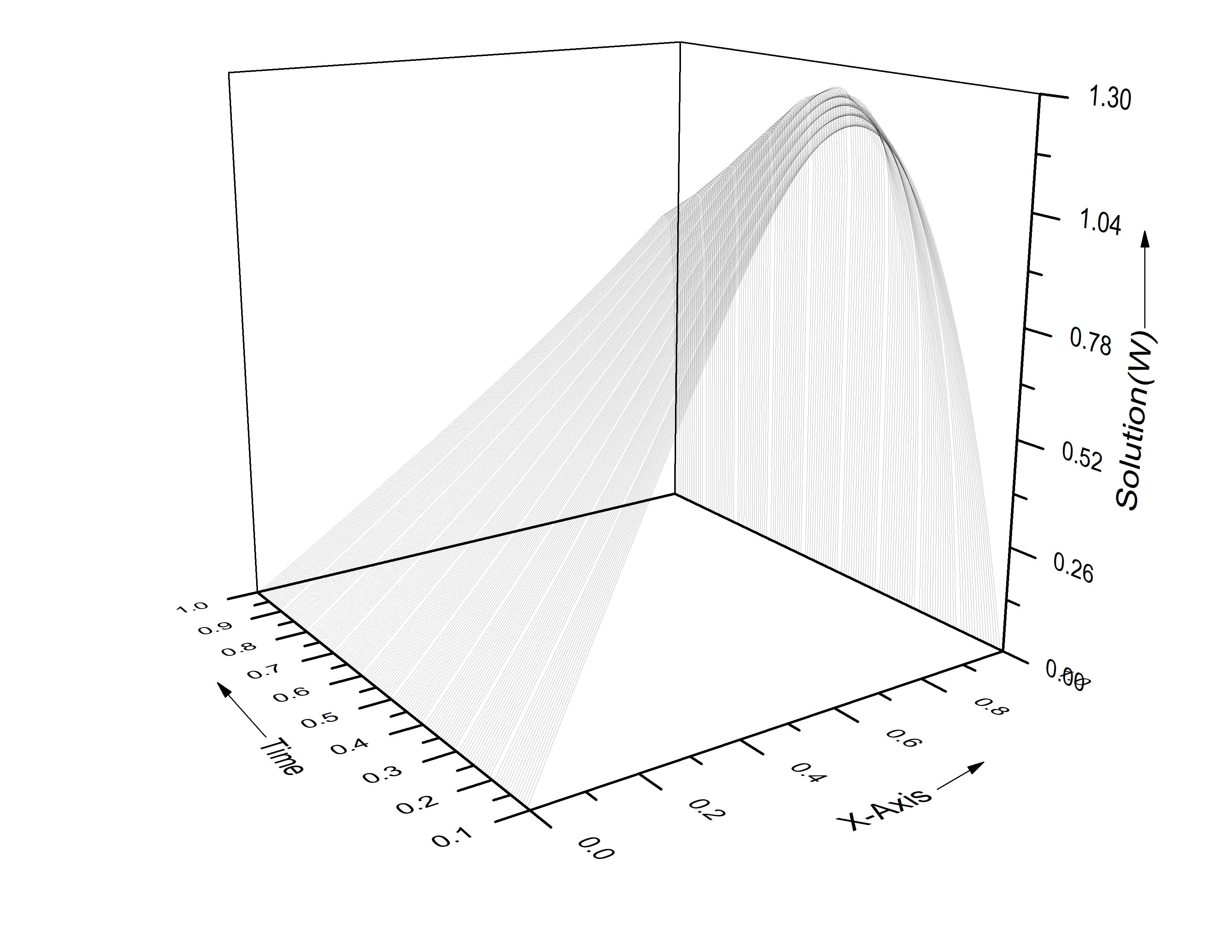}
		\caption{ Numerical solution for problem \ref{p2} for $\nu_{d}=0.01$, $\tau=0.001,~h=0.0125$ and at different times.}
		\label{fg8}
	\end{center}
\end{figure}

\begin{figure}[h!]
	\begin{center}
		\includegraphics[width=3in]{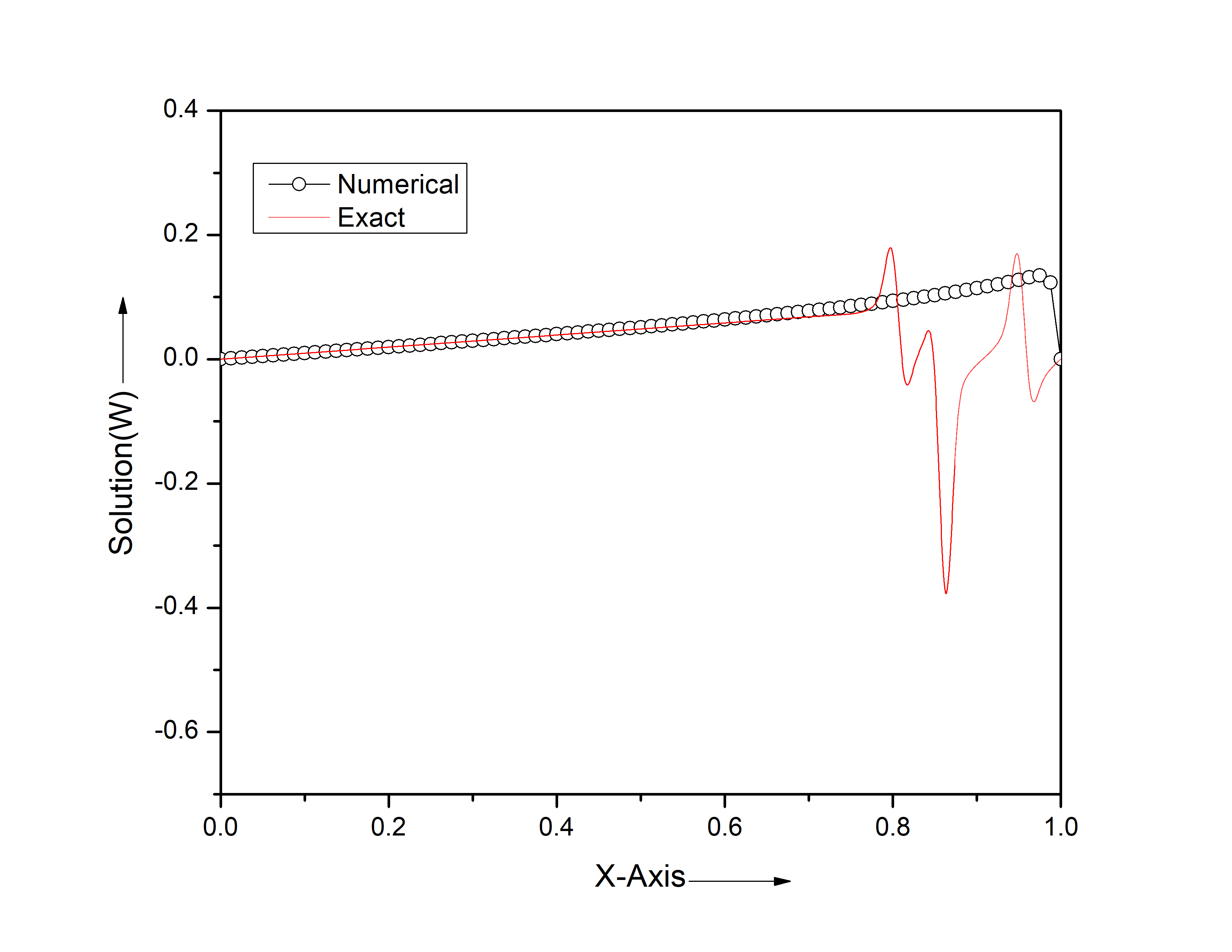}
		\caption{Comparison of Numerical and analytical solution for problem \ref{p2} at time $T=10$, $\nu_{d}=0.001,~\tau=0.001$ and $h=0.0125$.}
		\label{fg9}
	\end{center}
\end{figure}

\begin{figure}[h!]
	\begin{center}
		\includegraphics[width=3in]{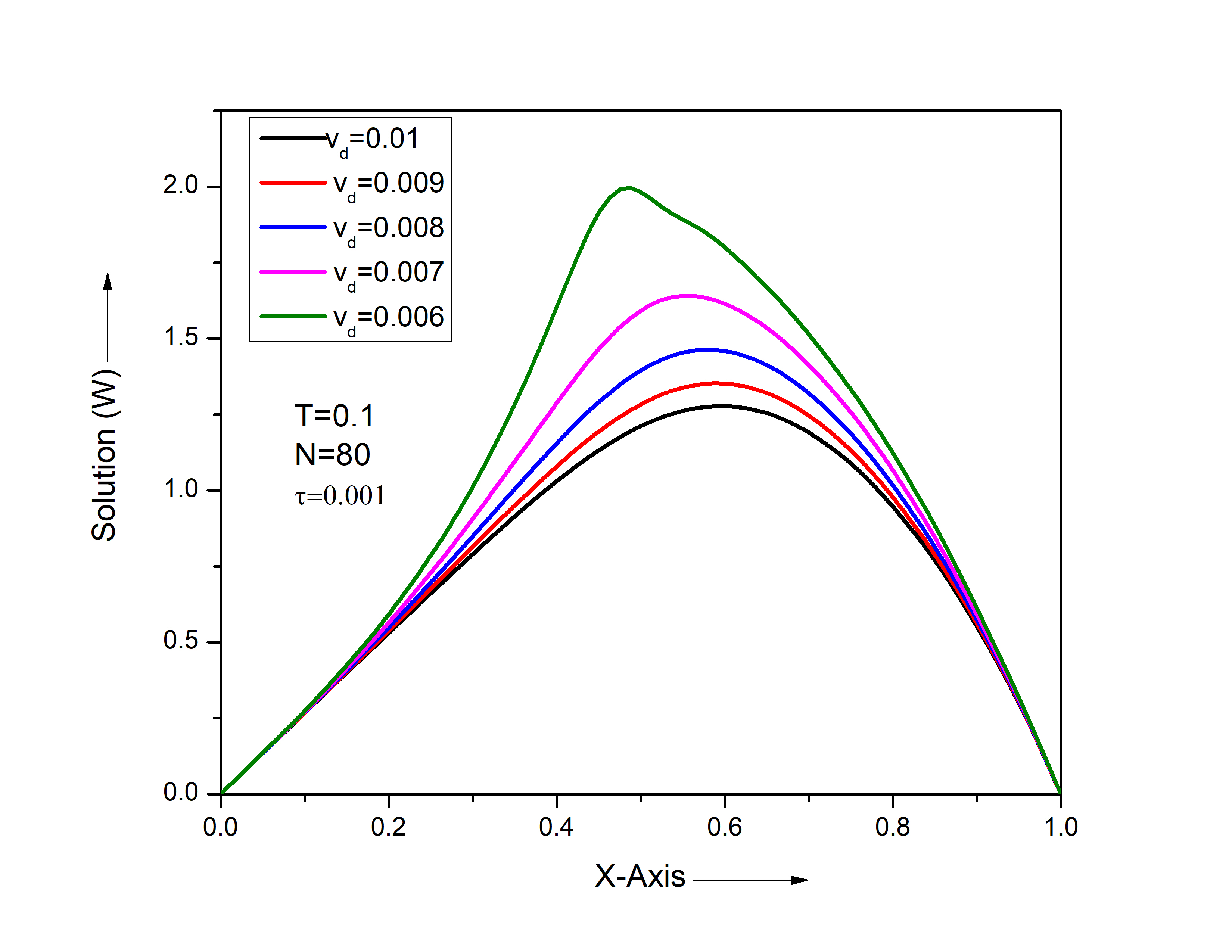}
		\caption{ Numerical solution for problem \ref{p2}  at time $T=0.1$, for different small value of $\nu_{d}$ with $\tau=0.001$ and $h=0.0125$.}
		\label{fg10}
	\end{center}
\end{figure}

\subsection{Example 3}
\label{p3}
Consider the shock-like solution of Burger's equation \cite{xie2008numerical}. The exact solution is given by 
\begin{equation}
\label{sw}
w(x,t)=\frac{\frac{x}{t}}{1+\sqrt{\frac{t}{t_{0}}e^{\frac{x^2}{2\nu_{d}t}}}},~~ x\in(0,1.2),~~t\geq 1,
\end{equation}
where $t_{0}=e^{\frac{1}{4\nu_{d}}}$ having boundary condition
\begin{equation}
w(0,t)=0=w(1.2,t),~~ t>1
\end{equation}
and initial condition
\begin{equation}
w(x,1)=\frac{x}{1+e^{\frac{1}{2\nu_{d}}(x^2-\frac{1}{4})}},~~x\in(0,1.2)
\end{equation}
which is obtained from equation \eqref{sw} by putting $t=1$.\par
For the comparison purpose in table \ref{tb7}, we take time step $\tau=0.01$, spatial step $h=0.0005$ and small value of viscous coefficient $\nu_{d}=0.002$. The numerical solution at different discrete point are compared with the exact solution and numerical solutions are also compared with the existing result in \cite{xie2008numerical}. It is found that the present method gives better result  compared to the method in  \cite{xie2008numerical}. For this example, discrete $L_{2}$ and $L_{\infty}$-error norm are also given and compared with the error for this example given in  \cite{xie2008numerical}. It can be seen that error produced by present method is far less than that of error by the method in  \cite{xie2008numerical}.  Figure \ref{fgp3} shows the correct physical behaviour of the present method for small value of $\nu_{d}=0.001.$

\begin{table}[!ht]
	\centering
	\caption{ Comparison of present result with existing result and it's error at different $T$  for $\nu_{d}=0.002,~h=0.0005$ and $\tau=0.01$  for problem \ref{p3}}
	\label{tb7}
	\scriptsize{
		\begin{tabular}{llllllllll}
			\hline
			$x$ & \multicolumn{3}{l}{T=1.7} \hspace{2cm}       & \multicolumn{3}{l}{T=3.0}     \hspace{1cm}    & \multicolumn{3}{l}{T=3.5}          \\ 
			& \multicolumn{3}{l}{\noindent\rule{4cm}{0.4pt}}        & \multicolumn{3}{l}{\noindent\rule{4cm}{0.4pt}}         & \multicolumn{3}{l}{\noindent\rule{4cm}{0.4pt}}\\
			& Exact                   & Present  & \cite{xie2008numerical}         & Exact    & Present  & \cite{xie2008numerical}         & Exact    & Present  & \cite{xie2008numerical}        \\ \hline
			0.2                                 & 0.117647                & 0.117660 & 0.11745      & 0.066667 & 0.066669 & 0.06648      & 0.057143 & 0.057144 & 0.05697      \\
			0.4                                 & 0.235294                & 0.235420 & 0.23456      & 0.133333 & 0.133355 & 0.13295      & 0.114286 & 0.114299 & 0.11394      \\
			0.6                                 & 0.352909                & 0.353346 & 0.34936      & 0.200000 & 0.200079 & 0.19922      & 0.171429 & 0.171478 & 0.17082      \\
			0.8                                 & 0.000000                & 0.000000 & 0.00000      & 0.266618 & 0.266808 & 0.26478      & 0.228571 & 0.228690  & 0.22737      \\
			1.0                                 & 0.000000                & 0.000000 & 0.00000      & 0.000000 & 0.000000 & 0.00000      & 0.000020 & 2.03E-05 & 0.000028     \\ \hline
			\multicolumn{2}{l}{$10^3\times L_{\infty}error$} & 0.50201 & 29.70447 &          & 0.21289 & 19.00976 &          & 0.16870 & 16.78871 \\
			\multicolumn{2}{l}{$10^3\times L_{2}error$}                     & 0.16675 & 3.59366  &          & 0.08135 & 2.63510  &          & 0.06695 & 2.41729 \\ \hline
	\end{tabular}}
\end{table}

\begin{figure}[h!]
	\begin{center}
		\includegraphics[width=3in]{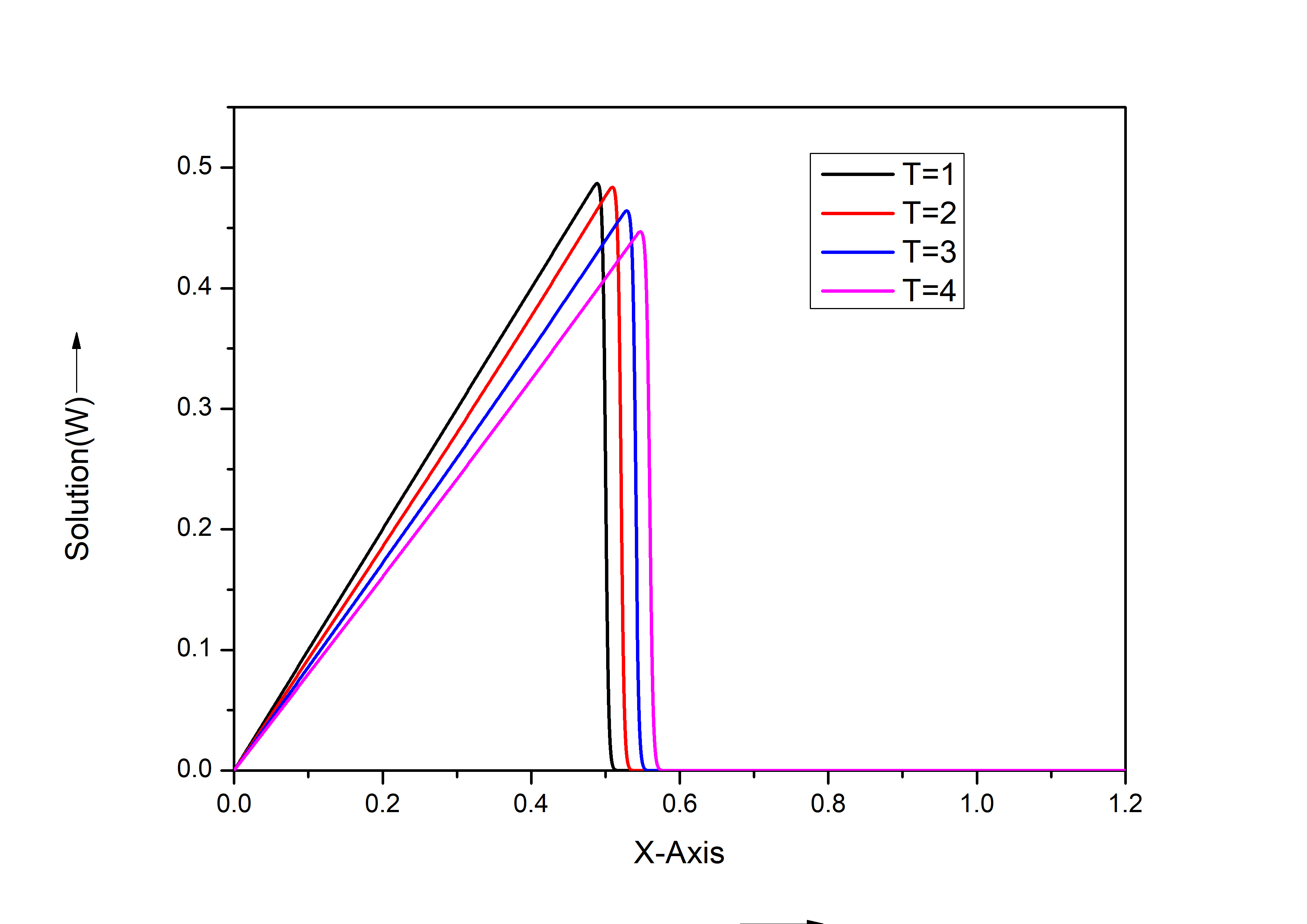}
		\caption{ Numerical solution for problem \ref{p3} for $\nu_{d}=0.001,h=0.001,\tau=0.01$ at different time $T$.}
		\label{fgp3}
	\end{center}
\end{figure}

\subsection{Example 4}
 \label{p4}
 In this example, we consider  the analytic solution of Burger's equation \cite{cecchi1996space} given by
 \begin{eqnarray}
 \label{be}
 w(x,t)=\pi \nu_{d}\frac{\sin(\pi x)exp(-\pi^2\nu_{d}^2t/4)+4\sin(2\pi x)exp(-\pi^2\nu_{d}^2t)}{4+\cos(\pi x)exp(-\pi^2\nu_{d}^2t/4)+2\cos(2\pi x)exp(-\pi^2\nu_{d}^2t)},
 \end{eqnarray}
 with the boundary condition $w(0,t)=0=w(2,t)$ and initial condition is obtained by putting $t=0$ in \eqref{be}.
 The physical behaviour of the numerical solution for $\nu_{d}=0.001,~\tau=0.01$ and $h=0.025$ is exhibited in the figure \ref{fgp4} (left). Absolute error are depicted in figure  \ref{fgp4} (right) and it can be seen that the absolute error are $\leq 10^{-3}$ for different time moment. Hence  the numerical result obtained by present method is acceptable.
 \begin{figure}[!ht]
 \includegraphics[height=2.5in,width=3.5in]{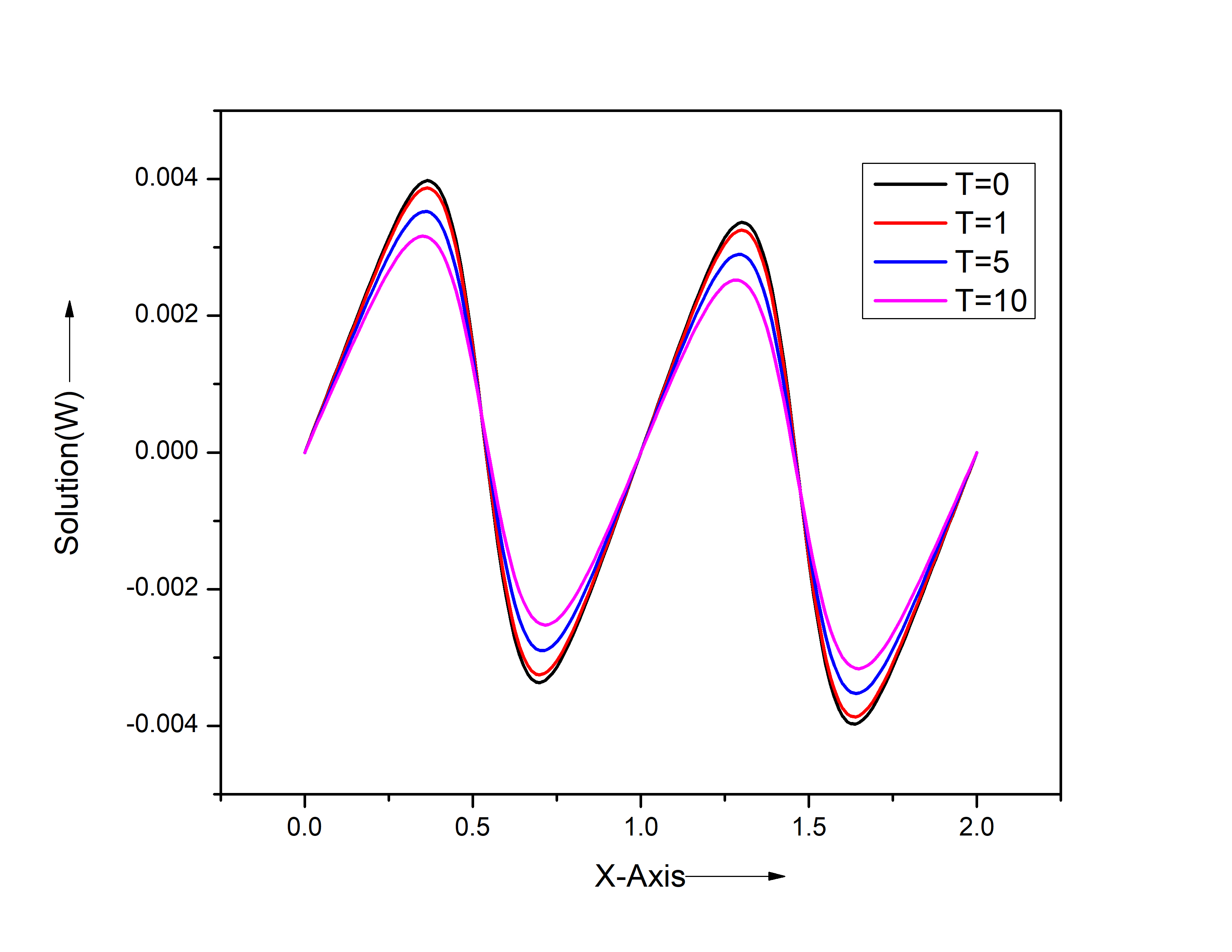}
 		\includegraphics[height=2.5in,width=3.5in]{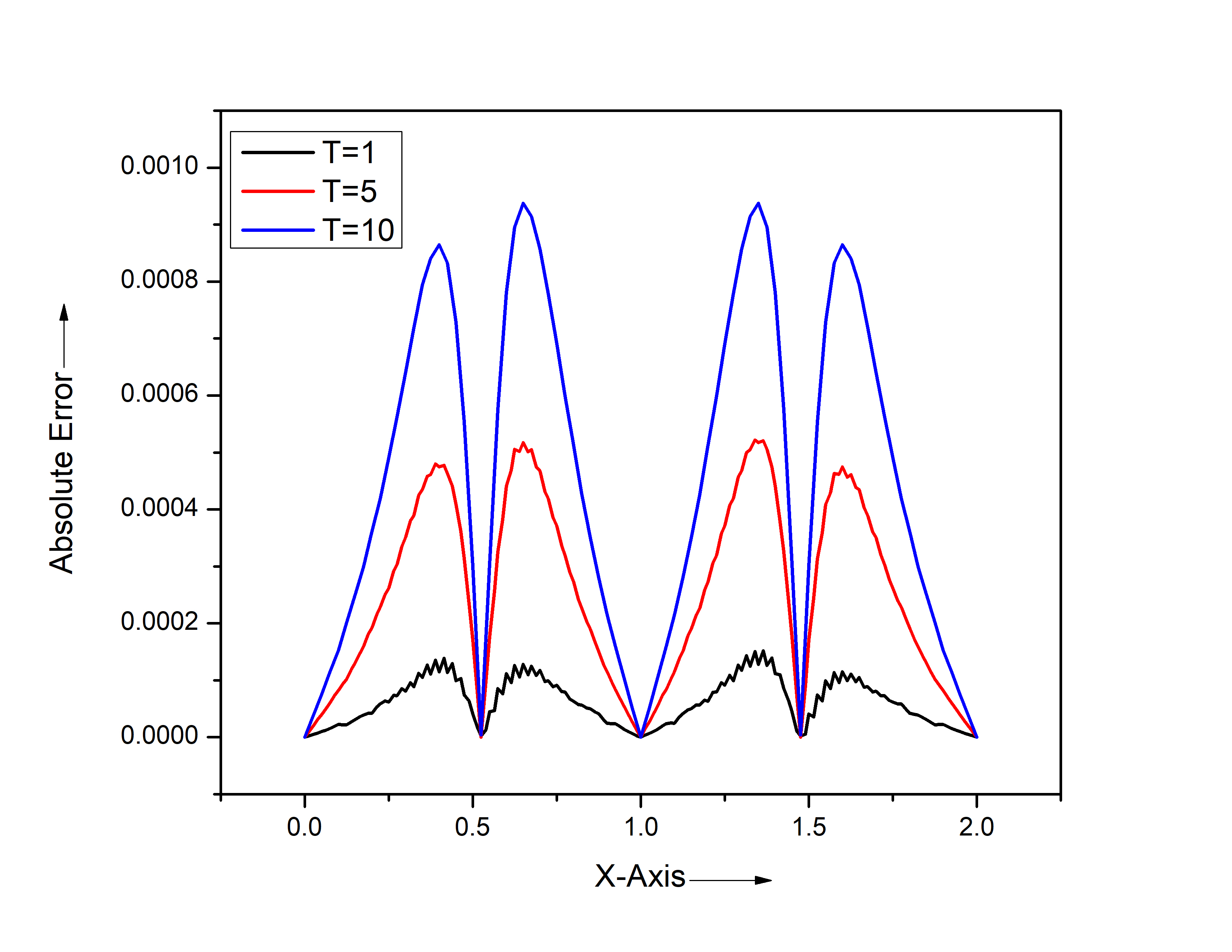}
 		\caption{Numerical approximation (left) and absolute error (right) of problem \ref{p4} for $\nu_{d}=0.001,h=0.025,\tau=0.01$ at different time $T$.}
 		\label{fgp4}
 	\end{figure}
\subsection{Example 5}
\label{l15}
Here, the BCs are same as \eqref{bc} and ICs as
\begin{eqnarray}
	w(x,0)=\sin \frac{\pi}{2}x,\hspace{.5cm} x\in(0,1).
\end{eqnarray}
Equation \eqref{ANBE} represents the analytical solution of the above problem, where 
\begin{eqnarray}
	&&\beta_{0}=\int_{0}^1 \exp\Big(\frac{2}{\pi \nu_{d}}(\cos\frac{\pi}{2} x-1)\Big)dx,\\
	&&\beta_{l}=2\int_{0}^1 \exp \Big(\frac{2}{\pi \nu_{d}}(\cos\frac{\pi}{2} x-1)\cos l\pi x\Big)dx,
\end{eqnarray}
are Fourier coefficients.\par
This example shows inconsistent ICs and BCs at the boundary point $1$. Figure \ref{fg11} shows high oscillation near the boundary point $1$ by the CN method  while the present method gives accurate and stable numerical solutions throughout the domain.
\begin{figure}[h!]
	\begin{center}
		\includegraphics[width=3in]{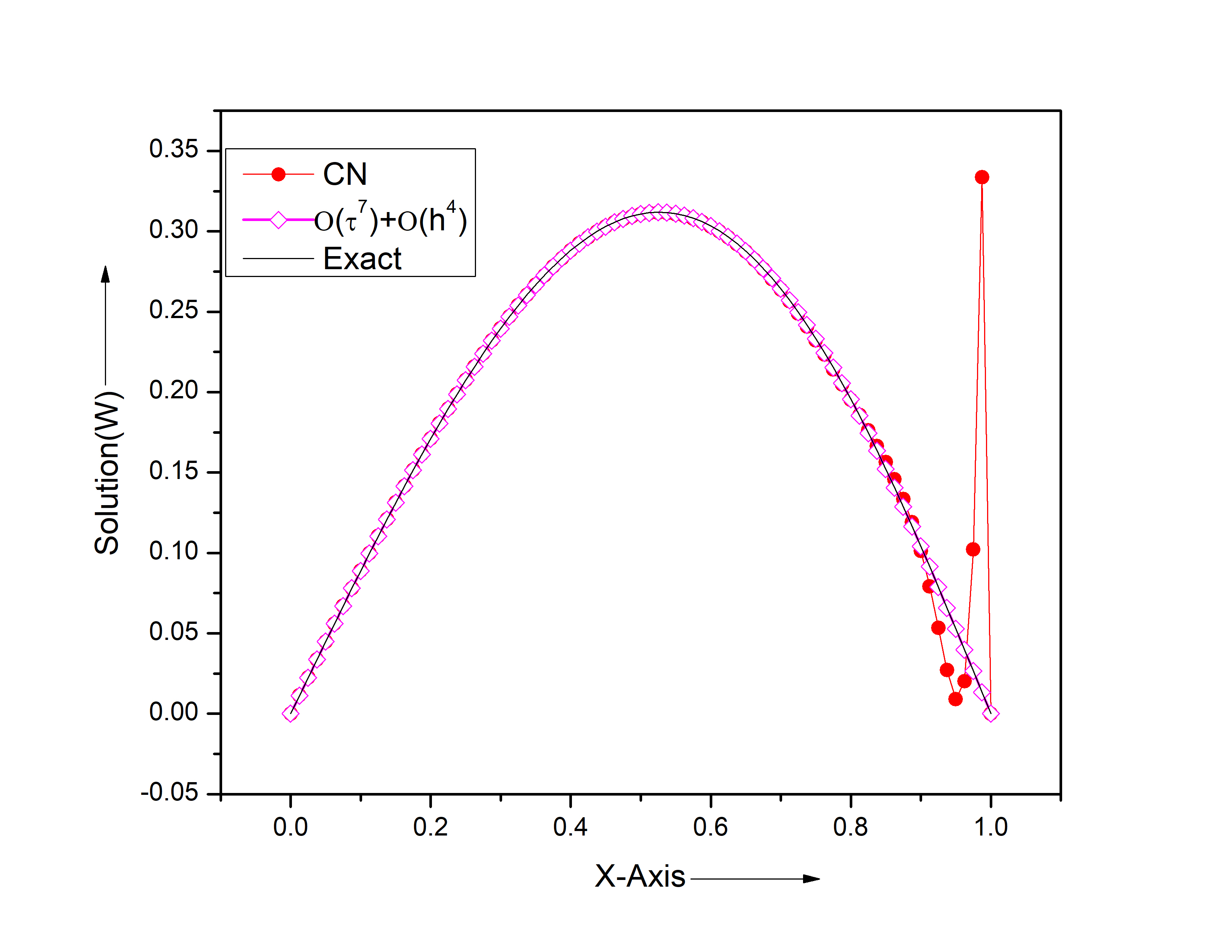}
		\caption{Comparison of  analytical solution with numerical solution by CN method and our method of the problem \ref{l15} at $T=0.1$ at different spacial point for $h=0.0125, \nu_{d}=2,$ and $\tau=0.01$.}
		\label{fg11}
	\end{center}
\end{figure}

\subsection{Example 6}
\label{l16}
Here, we take the BCs same as \eqref{bc} and ICs as
\begin{eqnarray}
	w(x,0)=\cos \frac{\pi}{4}x,\hspace{.5cm} x\in(0,1).
\end{eqnarray}
Equation \eqref{ANBE} represents the analytical solution of the above problem, where
\begin{eqnarray}
	&&\beta_{0}=\int_{0}^1 \exp\Big(-\frac{4}{\pi \nu_{d}}(sin\frac{\pi}{4} x)\Big)dx,\\
	&&\beta_{l}=2\int_{0}^1 \exp \Big(-\frac{4}{\pi \nu_{d}}(-\sin\frac{\pi}{4} x)\cos l\pi x\Big)dx,
\end{eqnarray}
are Fourier coefficient.\par
This example shows inconsistent ICs  and BCs at both the boundary point $0$ and $1$.  In figure \ref{fg12}, it can be seen that the numerical solution obtained by CN method has high oscillation near both the boundary point while the present method gives accurate and stable numerical solutions throughout the domain.
\begin{figure}[h!]
	\begin{center}
		\includegraphics[width=3in]{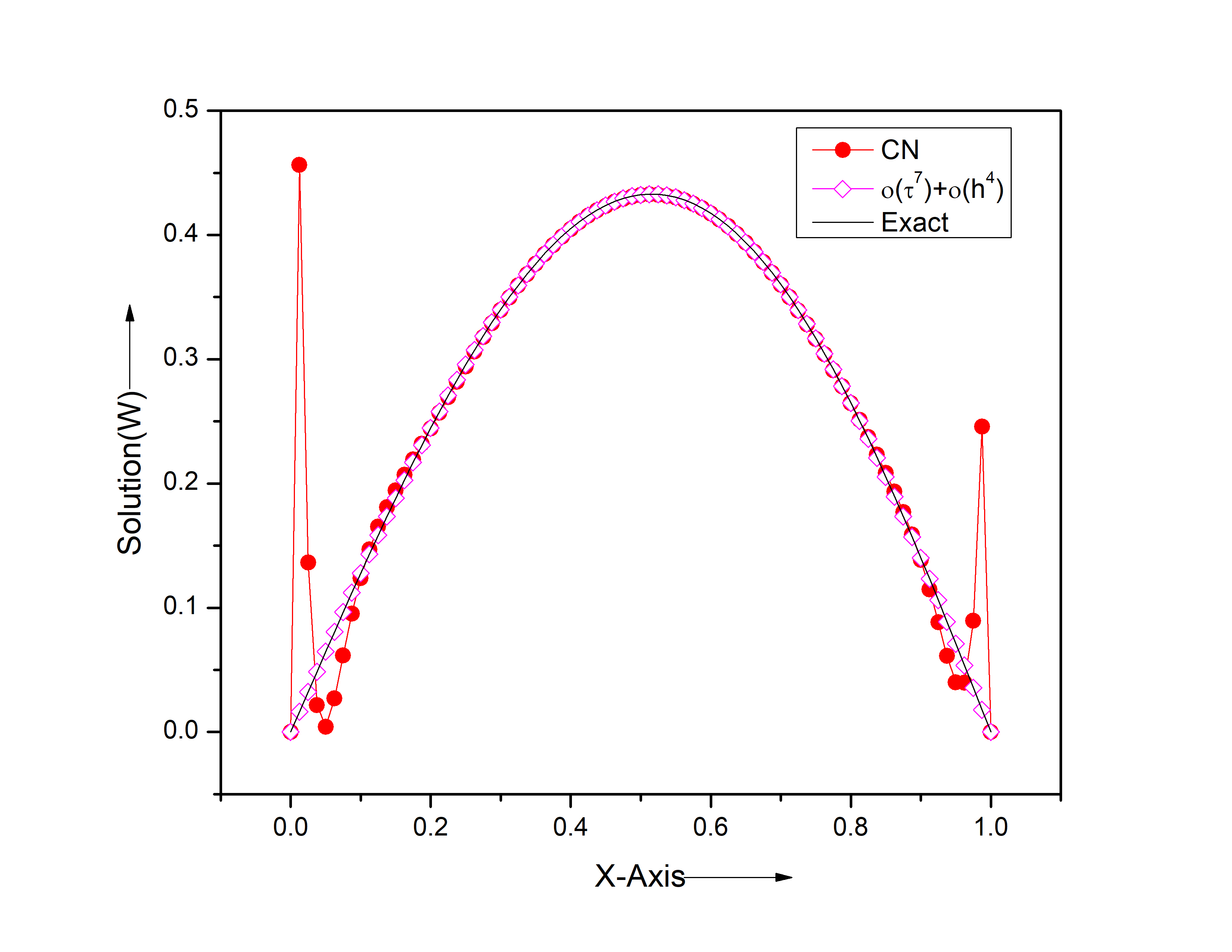}
		\caption{Comparison of analytical solution with numerical solution by CN method and our method of the problem \ref{l16}  at $T=0.1$  and at different spacial point for $h=0.0125, \nu_{d}=2,$ and $\tau=0.01$.}
		\label{fg12}
	\end{center}
\end{figure}

\section{Conclusion}
In the present paper, we have used $5^{th}$ order Hermite interpolation polynomial and $6^{th}$ order explict backward Taylor's series approximation formula to derive $7^{th}$ order time integration formula which is weakly L-stable. Also to linearize Burger's equation, we have used Hopf-Cole transformation and then $4^{th}$ order finite difference ratio for second order spatial derivative is used. Present method is tested over some problem and the approximated result obtained are satisfactory and comparabably good with the existing result found in literature. It is also observed that the numerical solutions are in good agreement with the  exact solutions for small values of viscosity. The strength of this method is that it is easy to implement and took very less time for computation. The ideas of the papers can further be extended to generalize the results further and study the behavior of the solution of the Burgers' equation for small values of viscosity.

\section{Acknowledgment} We are thankful to Dr Lajja, NSIT Bihta, IIT Patna, for her suggestions. 


\begin{thebibliography}{99}
	\bibitem{asaithambi2010numerical}
 A.	Asaithambi, {\em Numerical solution of the Burgers’ equation by automatic differentiation}, Applied Mathematics and Computation, 2010, 216(9), 2700--2708.
 
	\bibitem{aksan2006quadratic}
	EN Aksan, {\em Quadratic B-spline finite element method for numerical solution of the Burgers’ equation}, Applied Mathematics and Computation, 2006, 174(2), 884--896.
	
	\bibitem{ali1992collocation}
	AHA	Ali, GA Gardner and LRT Gardner, {\em A collocation solution of Burgers' equation using cubic b-spline finite element}, Computer Methods in Applied Mechanics and Engineering, 1992, 100(3), 325--337.
	
	\bibitem{bakodah2017decomposition}
	HO Bakodah, NA Al-Zaid, M Mirzazadeh  and Q Zhou, {\em Decomposition method for Solving Burgers’ Equation with Dirichlet and Neumann boundary conditions}, Optik, 2017, 130, 1339--1346.
	
	\bibitem{bateman1915some}
	H. Bateman, {\em Some recent researches on the motion of fluids}, Monthly Weather Review, 1915, 252(43), 163--170.
	
	\bibitem{burgers1939mathematical}
	JM Burgers, {\em Mathematical examples illustrating relations occurring in the theory of turbulent fluid motion, Verh}, Nederl. Akad. Wetensh. Afd. Wetensch. Afd. Natuurk. Sect, 1939, 1, 17.
	
	\bibitem{burgers1948mathematical}
	JM Burgers, {\em A mathematical model illustrating the theory of turbulence}, Computer mathematics in applied mechanics  and EngineeringAdvances in applied mechanics, 1948, 1, 171--199.
	\bibitem{caldwell1982solution}
	J. Caldwell  and P.  Smith, {\em Solution of Burgers' equation with a large Reynolds number}, Applied Mathematical Modelling, 1982, 6(5), 381--385.
	
	\bibitem{cecchi1996space}
  MM Cecchi, R Nociforo and PP Grego, {\em Space-time finite elements numerical solutions of Burgers Problems}, Le Matematiche, 1996, 51(1), 43--57.
  
   \bibitem{chawla1994stabilized}
   MM Chawla, MA Al-Zanaidi and MS Al-Sahhar, {\em Stabilized fourth order extended methods for the numerical solution of ODEs}, International journal of computer mathematics, 1994, 52(1-2), 99--107.
   
   \bibitem{chawla1999generalized}
   MM Chawla, MA Al-Zanaidi and DJ Evans, {\em Generalized trapezoidal formulas for parabolic equations}, International journal of computer mathematics, 1999, 70(3), 429--443.
   

   \bibitem{chawla2005new}
   MM Chawla  and DJ Evans,{\em A new L-stable Simpson-type rule for the diffusion equation}, International Journal of Computer Mathematics, 2005, 82(5), 601--607.
   
   \bibitem{cole1951quasi}
  JD  Cole, {\em On a quasi-linear parabolic equation occurring in aerodynamics}, Quarterly of applied mathematics, 1951, 9(3), 225--236.
  
  \bibitem{crank1947practical}
 J. Crank and  P. Nicolson, {\em A practical method for numerical evaluation of solutions of partial differential equations of the heat-conduction type}, Adv. Comput. Math, 1947, 6, 207--226.
 
 \bibitem{dogan2004galerkin}
 A. Dogan, {\em A Galerkin finite element approach to Burgers' equation}, Applied mathematics and computation, 2004, 157(2), 331--346.
 
 
\bibitem{elgindy2018high}
KT Elgindy, SA Dahy, {\em High-order numerical solution of viscous Burgers' equation using a Cole-Hopf barycentric Gegenbauer integral pseudospectral method}, Mathematical Methods in the Applied Sciences, 2018, 41(16), 6226--6251.

 \bibitem{fu2019moving}
F. Fu, J. Li, J.Lin, Y. Guan, F. Gao, C. Zhang, L. Chen, {\em Moving least squares particle hydrodynamics method for Burgers’ equation}, Applied Mathematics and Computation, 2019, 356, 362--378.

\bibitem{gowrisankar2019efficient}
S. Gowrisankar and S. Natesan, {\em An efficient robust numerical method for singularly perturbed Burgers’ equation}, Applied Mathematics and Computation, 2019, 346, 385--394.

\bibitem{hassanien2005fourth}
IA Hassanien, AA Salama, and HA Hosham, {\em Fourth-order finite difference method for solving Burgers’ equation}, Applied Mathematics and Computation 2005, 150(2), 781--800.

\bibitem{hopf1950partial}
E. Hopf, {\em The partial differential equation $u_{t}+ uu_{x}= \mu u_{xx}$}, Communications on Pure and Applied mathematics, 1950, 3(3), 201--230.

\bibitem{jiwari2013numerical}
R. Jiwari, RC Mittal  and KK Sharma, {\em A numerical scheme based on weighted average differential quadrature method for the numerical solution of Burgers’ equation}, Applied Mathematics and Computation, 2013, 219(12), 6680--6691.

\bibitem{jiwari2015hybrid}
R. Jiwari, {\em A hybrid numerical scheme for the numerical solution of the Burgers’ equation}, Computer Physics Communications, 2015, 188, 59--67.

\bibitem{khater2008chebyshev}
AH Khater, RS Temsah and MM Hassan, {\em A Chebyshev spectral collocation method for solving Burgers’-type equations}, Journal of Computational and Applied Mathematics, 2008, 222(2), 333--350.

\bibitem{korkmaz2011polynomial}
A. Korkmaz  and I. Dag, {\em Polynomial based differential quadrature method for numerical solution of nonlinear Burgers' equation}, Journal of the Franklin Institute, 2011, 348(10) 2863--2875.
 
 \bibitem{korkmaz2011quartic}
 A. Korkmaz,  AM Aksoy and I. Dag, {\em Quartic B-spline differential quadrature method}, Int. J. Nonlinear Sci, 2011, 11(4), 403--411.
\bibitem{kutluay1999numerical}
S. Kutluay, AR Bahadir, and A. {\"O}zde{\c{s}}, {\em Numerical solution of one-dimensional Burgers equation: explicit and exact-explicit finite difference methods}, Journal of Computational and Applied Mathematics, 1999, 103(2), 251--261.

\bibitem{kutluay2004numerical}
S. Kutluay, A. Esen  and I. Dag, {\em Numerical solutions of the Burgers’ equation by the least-squares quadratic B-spline finite element method}, Journal of Computational and Applied Mathematics, 2004, 167(1), 21--33.

\bibitem{lawson1978extrapolation}
JD Lawson and JL  Morris, {\em The extrapolation of first order methods for parabolic partial differential equations}, SIAM Journal on Numerical Analysis, 1978, 15(6), 1212--1224.

\bibitem{LeVeque2007}
R.J. LeVeque, {\em Finite Difference Methods for Ordinary and Partial Differential Equations}, SIAM, Philadelphia, 2007.

\bibitem{lukyanenko2018solving}
DV Lukyanenko, MA Shishlenin  and VT Volkov, {\em Solving of the coefficient inverse problems for a nonlinear singularly perturbed reaction-diffusion-advection equation with the final time data}, 2018, 54, 233--247.

\bibitem{mittal1993numerical}
RC Mittal, and P. Singhal, {\em Numerical solution of Burger's equation}, Communications in numerical methods in engineering, 1993, 9(5), 397--406.

\bibitem{mittal2009differential}
RC Mittal and R. Jiwari, {\em Differential quadrature method for two-dimensional Burgers' equations}, International Journal for Computational Methods in Engineering Science and Mechanics, 2009, 10(6), 450--459.

\bibitem{mittal2012differential}
RC Mittal and R. Jiwari, {\em Differential quadrature method for numerical solution of coupled viscous Burgers’ equations}, International Journal for Computational Methods in Engineering Science and Mechanics, 2012, 13(2), 88--92.

\bibitem{morton1967difference}
KW Morton, {\em Difference Methods for Initial-Value Problems}, John Wiley \& Sons, 1967.

\bibitem{ozics2003finite}
T. {\"O}zi{\c{s}},  EN  Aksan and A. {\"O}zde{\c{s}}, {\em A finite element approach for solution of Burgers’ equation}, Applied Mathematics and Computation, 2003, 139(2-3), 417--428.

\bibitem{ryu2019improved}
S. Ryu, G. Lyu, Y. Do,   and GW Lee, {Improved rainfall nowcasting using Burgers’ equation}, Journal of Hydrology, 2019, 124-140.

\bibitem{saka2007quartic}
B. Saka, and I. Da{\u{g}}, {\em Quartic B-spline collocation method to the numerical solutions of the Burgers’ equation}, Chaos, Solitons \& Fractals ,2007, 32(3), 1125--1137.

\bibitem{seydaouglu2018accurate}
M. Seydao{\u{g}}lu, {\em An accurate approximation algorithm for Burgers’ equation in the presence of small viscosity}, Journal of Computational and Applied Mathematics, 2018, 344, 473--481.

\bibitem{seydaouglu2019meshless}
M. Seydao{\u{g}}lu, {\em A Meshless Method for Burgers’ Equation Using Multiquadric Radial Basis Functions with a Lie-Group Integrator}, Mathematics, 2019, 7(2), 13.

\bibitem{shiralashetti2019numerical}
SC Shiralashetti, LM Angadi,{\em Numerical solution of Burgers' equation using Biorthogonal wavelet-based full approximation scheme}, International Journal of Computational Materials Science and Engineering, 2019, DOI:	10.1142/S2047684118500306. 

\bibitem{smith1978numerical}
GD Smith, {\em Numerical solution of partial differential equations: Finite difference methods: Oxford University Press}, Oxford, 1978.

\bibitem{thomas2013numerical}
JW Thomas, {\em Numerical partial differential equations: finite difference methods}, Springer Science \& Business Media, 2013.
\bibitem{verma2015higher}
AK Verma  and  L. Verma, {\em Higher order time integration formula with application on Burgers’ equation}, International Journal of Computer Mathematics, 2015, 92(4), 756--771.

\bibitem{verma2012stable}
L. Verma, {\em L-Stable Derivative-Free Error-Corrected Trapezoidal Rule for Burgers' Equation with Inconsistent Initial and Boundary Conditions}, International Journal of Mathematics and Mathematical Sciences, 2012, DOI: dx.doi.org/10.1155/2012/821907. 

\bibitem{KPLVAKVAJOM} K. Pandey, Lajja Verma, Amit K. Verma, {\em Du Fort-Frankel finite difference scheme for Burgers equation}, Arabian Journal of Mathematics, 2013, Volume 2, Issue 1, pp 91--101.

\bibitem{AMC2009} K. Pandey, Lajja Verma, Amit K. Verma, {\em On a finite difference scheme for Burgers' equation}, Applied Mathematics and Computation, Applied Mathematics and Computation, 2009, Volume 215, Issue 6, 2206--2214.


\bibitem{xu2011novel}
M. Xu, RH Wang, JH Zhang, and Q. Fang, {\em A novel numerical scheme for solving Burgers’ equation}, Applied mathematics and computation, 2011, 217(9), 4473--4482.

\bibitem{xie2008numerical}
SS Xie, S. Heo, S. Kim, G. Woo,  and S. Yi, {\em Numerical solution of one-dimensional Burgers’ equation using reproducing kernel function}, Journal of Computational and Applied Mathematics, 2008, 214(2), 417--434.
	\end{thebibliography}
\end{document}